\documentclass[11pt]{article}
\usepackage{amssymb}
\usepackage{amsfonts}
\usepackage{amsmath,tabu}
\usepackage{amsthm}
\usepackage{cite}
\usepackage[colorlinks=true,citecolor=red,linkcolor=blue,pdfpagetransition=Blinds]{hyperref}
\usepackage{enumerate}

\usepackage{caption}

\usepackage[cm]{fullpage}

\newcommand{\eps}{{\varepsilon}}

\newcommand{\R}{\mathbb{R}}

\newcommand{\N}{\mathbb{N}}

\newcommand{\cuad}{{\sqcap\kern-.68em\sqcup}}

\numberwithin{equation}{section}
\newtheorem{theorem}{Theorem}[section]

\newtheorem{proposition}[theorem]{Proposition}

\newtheorem{lemma}[theorem]{Lemma}
\newtheorem{corollary}[theorem]{Corollary}
\newtheorem{remark}[theorem]{Remark}
\newcommand{\bremark}{\begin{remark} \em}
\newcommand{\eremark}{\end{remark} }

\newcommand{\cE}{{\mathcal E}}

\newcommand{\cH}{{\mathbb H}}

\newcommand{\cL}{{\mathcal L}}

\usepackage{graphicx}

\def\sideremark#1{\ifvmode\leavevmode\fi\vadjust{\vbox to0pt{\vss
 \hbox to 0pt{\hskip\hsize\hskip1em
 \vbox{\hsize2.1cm\tiny\raggedright\pretolerance10000
  \noindent #1\hfill}\hss}\vbox to15pt{\vfil}\vss}}}%

\usepackage{color}
\usepackage[dvipsnames]{xcolor}

\DeclareMathOperator{\dist}{\rm dist}

\title{Optimal boundary regularity and a Hopf-type lemma for Dirichlet problems involving the logarithmic Laplacian}
\author{Víctor Hernández-Santamaría\thanks{The work of V. Hern\'andez-Santamar\'ia is supported by the program ``Estancias Posdoctorales por México para la Formación y Consolidación de las y los Investigadores por México'' of CONAHCYT (Mexico). He also received support from Projects A1-S-17475 and A1-S-10457 of CONAHCYT and by UNAM-DGAPA-PAPIIT grants IN109522, IN104922, and IA100324 (Mexico).} \and Luis Fernando López Ríos\footnote{L.F. López Ríos is supported by CONAHCYT, México, grant CF-2023-G-122.} \and Alberto Salda\~{n}a\footnote{ A. Saldaña is supported  by 
CONAHCYT grants CBF2023-2024-116 and A1-S-10457 (Mexico) and by UNAM-DGAPA-PAPIIT grant IA100923 (Mexico).
}
}

\date{}

\begin{document}

\maketitle

\abstract{
We study the optimal boundary regularity of solutions to Dirichlet problems involving the logarithmic Laplacian. Our proofs are based on the construction of suitable barriers via the Kelvin transform and direct computations.  As applications of our results, we show a Hopf-type lemma for nonnegative weak solutions and the uniqueness of solutions to some nonlinear problems.
}

\bigskip
\bigskip

\medskip

\medskip

\section{Introduction}

In this paper, we study the optimal boundary regularity of solutions of
\begin{align}\label{DP}
    L_\Delta u = f\quad \text{ in }\Omega,\qquad u=0\quad \text{ in }\R^N \setminus \Omega,
\end{align}
where $\Omega\subset \R^N$ ($N\geq 1$) is a bounded open set and $f\in L^\infty(\Omega)$.  Here $L_\Delta$ denotes the logarithmic Laplacian, that is, the pseudodifferential operator with Fourier symbol $2\ln|\xi|$.  This operator can also be seen as a first-order expansion of the fractional Laplacian (the pseudodifferential operator with symbol $|\xi|^{2s}$); in particular, for $\varphi\in C^\infty_c(\R^N)$,
\begin{align*}
(-\Delta)^s\varphi = \varphi + sL_\Delta \varphi + o(s)\qquad \text{as $s\to 0^+$ in }L^p(\R) \text{ with }1<p\leq \infty,
\end{align*}
see \cite[Theorem 1.1]{CW19}.  Moreover, it has also the following integral representation
\begin{align}\label{loglap:def}
L_{\Delta}u(x):=c_{N}\int_{B_{1}(x)}\frac{u(x)-u(y)}{|x-y|^{N}}\, dy-c_{N}\int_{\mathbb{R}^{N}\setminus B_{1}(x)}\frac{u(y)}{|x-y|^{N}}\, dy+\rho_{N}u(x),
\end{align}
where $c_N$ and $\rho_N$ are explicit constants given in~\eqref{rhon} and $u$ is a suitable function.  The logarithmic Laplacian appears in some interesting applications.  For instance, it is a tool to characterize the differentiability properties of the solution mapping of fractional Dirichlet problems, see \cite{JSW20,JSW23}.  It is also used to describe the behavior of solutions to linear and nonlinear problems involving the fractional Laplacian $(-\Delta)^s$ in the small order limit (as $s\to 0$), see \cite{FJW22,AS22,hernandez2022small}, where the operator $L_\Delta$ naturally appears.  We also mention the following very recent works: the Cauchy problem with $L_\Delta$ is studied in \cite{chen2023cauchy}; higher-order expansions of the fractional Laplacian are considered in \cite{chen2023taylor}; a characterization of the logarithmic Laplacian via a local extension problem on the $(N+1)$-dimensional upper half-space in the spirit of the Caffarelli-Silvestre extension is obtained in \cite{CHW23}; and a finite element method is analyzed and implemented in \cite{HJSS23} to approximate solutions of~\eqref{DP}.

The kernel in~\eqref{loglap:def} is sometimes called of \emph{zero-order}, because it is a limiting case for hypersingular integrals.  As a consequence, the regularizing properties of this type of operators are very weak and are a subject of current research.  Interior regularity of (bounded weak) solutions to~\eqref{DP} has been studied in \cite{CW19,CS22,feulefack2021nonlocal,KM17}; in particular, it is known that if $f\in L^\infty(\Omega)$ then $u\in C(\overline{\Omega})$.  Furthermore, the modulus of continuity of $u$ can be characterized: $u$ belongs to the space of $\alpha$-log-Hölder continuous functions in $\R^N$ for some $\alpha\in(0,1)$ (see the definition of $\mathcal L^\alpha(\R^N)$ in~\eqref{log:H:def} below); in particular,
\begin{align}\label{reg:at:bdry}
 \|u\|_{\mathcal L^\alpha(\R^N)}=\|u\|_{L^\infty(\R^N)}+\sup_{\substack{x,y\in \R^N\\x\neq y}}\frac{|u(x)-u(y)|}{\ell^\alpha(|x-y|)}<\infty,
\end{align}
where $\ell:[0,\infty)\to[0,\infty)$ is given by 
\begin{align*}
\ell(r):=\frac{1}{|\ln(\min\{r,0.1\})|},    
\end{align*}
see Figure~\ref{f} below. 

This regularity is not enough to evaluate $L_\Delta u$ pointwisely; however, if the right-hand side $f$ is 1-log-Hölder continuous in $\Omega$, then $u$ is $(1+\alpha)$-log-Hölder continuous in $\Omega$ and, with this regularity, the equation~\eqref{DP} holds pointwisely, that is, $u$ is a \emph{classical solution} (see \cite[Theorem 1.1]{CS22}). 

It is an interesting question to determine what is the optimal $\alpha$ that characterizes the regularity of a solution across the boundary of the domain $\partial \Omega$, namely, the largest $\alpha$ so that~\eqref{reg:at:bdry} holds.  In this regard, in \cite[Theorem 1.11]{CW19} it has been established that, if $\Omega$ is a Lipschitz domain with uniform exterior sphere condition, $f\in L^\infty(\Omega)$, and $u$ is a (bounded weak) solution of~\eqref{DP}, then there is $C>0$ such that
\begin{align}\label{bdr:est}
    |u(x)|\leq C\ell^{\tau}(\operatorname{dist}(x,\partial \Omega))\qquad \text{ for all }x\in \Omega \text{ and for all }\tau\in(0,\tfrac{1}{2}).
\end{align}
The proof of this estimate is done by constructing suitable barriers via direct computations and using comparison principles in small domains. Another related result on boundary regularity is given in \cite{kim2014green}, where, using probabilistic tools, sharp two-sided boundary estimates are shown for Green functions in bounded $C^{1,1}$-domains associated to a subordinate Brownian motion $X$ when the Laplace exponent of the corresponding subordinator is a Bernstein function. These results yield boundary estimates for problems that are closely related to~\eqref{DP}, for example, they apply to the logarithmic Shrödinger operator $(I-\Delta)^{\log}$, which is the pseudodifferential operator with symbol $\ln(1+|\xi|^2)$, see \cite{feulefack2023logarithmic,feulefack2023fractional} and the references therein.  For this operator, the results in \cite{kim2014green} would yield that~\eqref{bdr:est} holds with $\tau=\frac{1}{2}$.  This strongly suggests that the same should be true for the operator $L_\Delta$; however, the results in \cite{kim2014green} cannot be applied directly to $L_\Delta$ (note that the symbol $2\ln|\xi|$ is negative for $|\xi|<1$ and tends to $-\infty$ as $\xi\to 0$).

Our main results show that~\eqref{bdr:est} also holds for $\tau=\frac{1}{2}$ and that this value is optimal.  Let us introduce some notation.  As usual, $B_r(x)$ denotes the open ball in $\R^N$ of radius $r>0$ centered at $x$ and $B_r:=B_r(0)$.  We say that a bounded open set $\Omega\subset \R^N$ satisfies a \emph{(uniform) exterior sphere condition} if there is $\delta>0$ such that, for all $x_0 \in \partial \Omega$,
\begin{align}\label{eta:def}
\text{there is $y_0\in \R^N$ with $\overline{B_{\delta}(y_0)} \cap \overline{\Omega} = \{ x_0 \}$.}
\end{align}
  We use $\mathbb H(\Omega)$ to denote the Hilbert space associated to $L_\Delta$ with Dirichlet exterior conditions, see~\eqref{Hdef}. 
  
\begin{theorem}\label{convex:thm}
Let $\Omega\subset \R^N$ be a bounded open set satisfying an exterior uniform sphere condition, let $f \in L^\infty(\Omega)$, and let $u\in {\mathbb H}(\Omega)\cap L^\infty(\R^N)$ be a weak solution of~\eqref{DP}. Then $u\in C(\R^N)$ and there is $C>0$ such that
\begin{equation}\label{eq:boundary_bound}
    |u(x)|\leq C\ell^{\frac{1}{2}}(\operatorname{dist}(x,\partial \Omega))\qquad \text{ for all }x\in \Omega.
\end{equation}    
\end{theorem}
This result is the analogue of \cite[Lemma 2.7]{RS14} for the logarithmic Laplacian.

By studying the case of the torsion function in a (small) ball, we also show that this boundary regularity in Theorem~\ref{convex:thm} is optimal. For $A\subset \R^N$, we use $|A|$ to denote its Lebesgue measure. Let $\Psi$ denote the digamma function and let $\gamma$ be the Euler-Mascheroni constant (see the Notation section below).

\begin{theorem} \label{torsion:ball}
Let $N\geq 1,$ $r>0$ be such that $|B_r|<2^N e^{\frac{N}{2}(\Psi(\frac{N}{2})-\gamma)}|B_1|,$ and let $\tau$ be the torsion function of the ball $B_r$, namely, the unique classical solution of 
\begin{align}\label{tau:eq}
    L_\Delta \tau = 1\quad \text{ in }B_r,\qquad \tau=0\quad \text{ in }\R^N \backslash B_r.
\end{align}
Then,
\begin{align}\label{liminf}
\infty>\liminf_{t\to 0^+} \frac{\tau(x_0-t\eta(x_0))}{\ell^\frac{1}{2}(t)}>0\qquad \text{ for }x_0\in\partial B_r,
\end{align}
where $\eta$ is the outer unit normal vector field along $\partial B_r$.  Moreover, there is $c>1$ such that
\begin{align}\label{tau:bds}
c^{-1}\ell^{\frac{1}{2}}(r-|x|)\leq \tau(x)\leq c\,\ell^{\frac{1}{2}}(r-|x|)\quad \text{ for }x\in B_r.
\end{align}
\end{theorem}
We refer to \cite[Figure 1]{HJSS23} for the numerical approximation of the torsion function in different intervals. 

The proofs of Theorems~\ref{convex:thm} and~\ref{torsion:ball} are done by constructing suitable barriers and by using a comparison principle in small domains. However, in the case of the logarithmic Laplacian, this is a highly nontrivial task, because there is no easy choice for a suitable barrier; in particular, a closed formula for the torsion function in any ball is not available and one cannot argue as in \cite{RS14} (this is also the reason why~\eqref{bdr:est} is only established for $\tau\in(0,\frac{1}{2})$ in \cite{CW19}).  To overcome this difficulty, we use the following strategy.  First, we construct a barrier in an open interval $I=(0,2)$ that behaves as $\ell^\frac{1}{2}(x)$ for $x$ close to 0 (see~\eqref{u}).  Then, via sharp direct computations, we show that its logarithmic Laplacian is bounded in $I$ and that it is positive close to 0, see Theorem~\ref{thm:L}. Afterwards, we use this barrier to construct higher-dimensional barriers in half balls (see Figure~\ref{fig1} and~\eqref{Vtilde}).  Using the calculations in the one-dimensional case and a Leibniz-type formula (see Section~\ref{Leib:sec}), we show that this new function also has a bounded logarithmic Laplacian that is positive close to the origin (see Theorem~\ref{thm:log:N}).

 \begin{figure}[ht]
     \centering
 \includegraphics[width=.32\paperwidth]{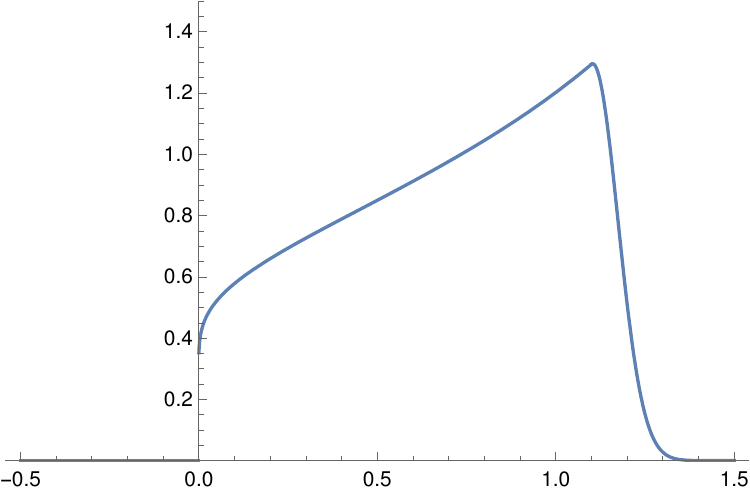} 
 \includegraphics[width=.42\paperwidth]{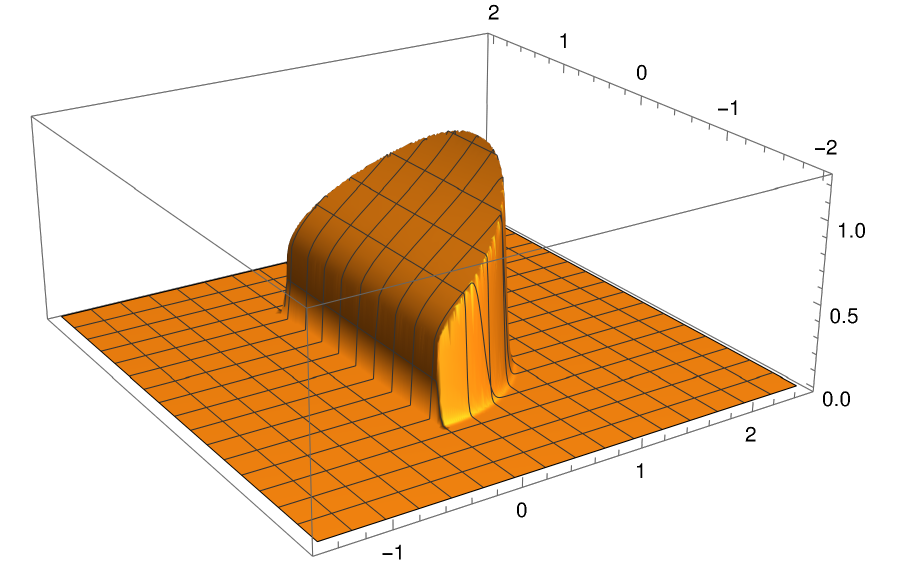} 
     \caption{\small The shape of the initial barriers for $N=1$ and $N=2$.}
     \label{fig1}
     \end{figure}
     
Finally, to obtain a suitable barrier to characterize domains satisfying the (uniform) exterior sphere condition, we use the following Kelvin transform formula, which is of independent interest. Let $x,x_0\in \R^N$ and $R>0$, the inversion of $x$ with respect to the sphere $S_{R}(x_0):=\partial B_R(x_0)$ is given by
\begin{equation*}
    x^* = x_0 + R^2\frac{x-x_0}{|x-x_0|^2}, \quad x \neq x_0,
\end{equation*}
and the Kelvin transform of a function $u:\R^N\to \R$ with respect to the sphere $S_{R}(x_0)$ is
\begin{equation*}
    u^*(x^*) = |x^*-x_0|^{-N} u(x),\qquad x^* \neq x_0.
\end{equation*}
For $\Omega \subset \R^N$ we define $\Omega^* = \{ x^*: x\in \Omega \}$. 

\begin{proposition}\label{K:lem}
Let $x_0\in \R^N$, $R>0$, and $x^*$ denote the inversion of $x$ with respect to the sphere $S_{R}(x_0)$. 
 Let $\Omega$ be an open bounded Lipschitz set and $u: \R^N \to \R$ be a measurable function with $u \equiv 0$ on $\R^N \setminus \Omega$. 
\begin{enumerate}[i)]
\item If $u$ is Dini continuous in $x \neq x_0$, then $u^*$ is Dini continuous in $x^*$ and
\begin{equation} \label{eq:L_delta1}
L_\Delta u^*(x^*)= |x^*|^{-N} L_\Delta u(x) + c_N |x^*|^{-2N} u(x) \int_{\Omega}\frac{|z|^{-N}-|x|^{-N}}{|z-x|^N}\, dz  + (h_{\Omega^*}(x^*)-h_\Omega(x))|x^*|^{-N}u(x),
\end{equation}
where $h_\Omega$ is given in~\eqref{eq:h_omega}.   
\item If $u, L_\Delta u \in L^\infty(\Omega)$ and $B_R(x_0) \subset \Omega^c$, then $u^*, L_\Delta u^* \in L^\infty(\Omega^*)$.
\item If $B_R(x_0) \subset \Omega^c$ and $u \in \mathbb{H}(\Omega)$, then $u^* \in \mathbb{H}(\Omega^*)$.
\end{enumerate}
\end{proposition}
Note that formula~\eqref{eq:L_delta1} is different from the one obtained for the fractional Laplacian (see, for example, \cite{BZ06} or \cite[Proposition 2]{ADFJS19}).   For the logarithmic Laplacian, the geometry of the domain has a stronger influence in the formula, and this is represented in the last two summands in~\eqref{eq:L_delta1}.

Using the initial barrier and the Kelvin transform, we obtain a new function which has a bounded logarithmic Laplacian and which has the optimal regularity at the curved part of the boundary (Proposition~\ref{v:barrier:prop}), see Figure~\ref{fig:kelvin}. 

\begin{figure}[ht]
    \centering
    \includegraphics[width=5cm]{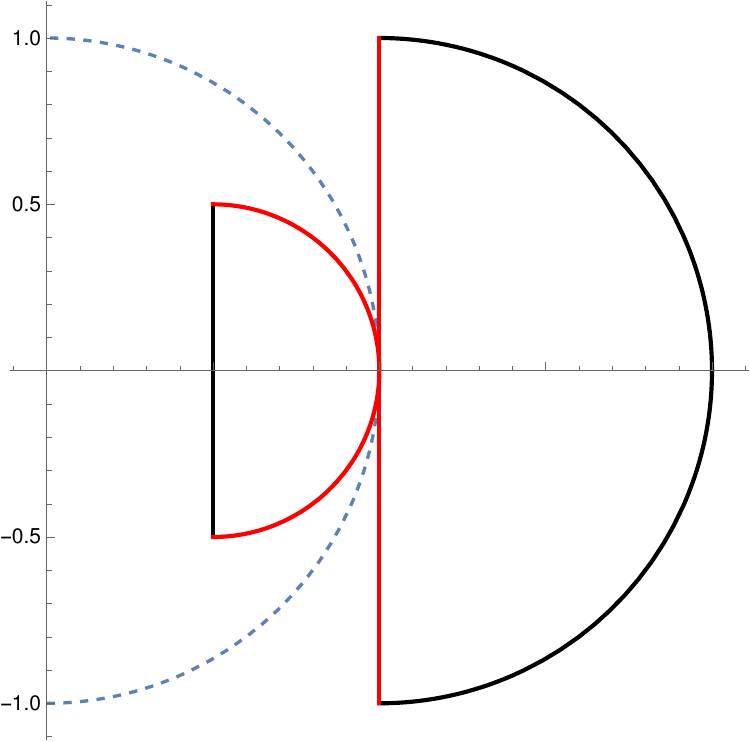}
    \includegraphics[width=8cm]{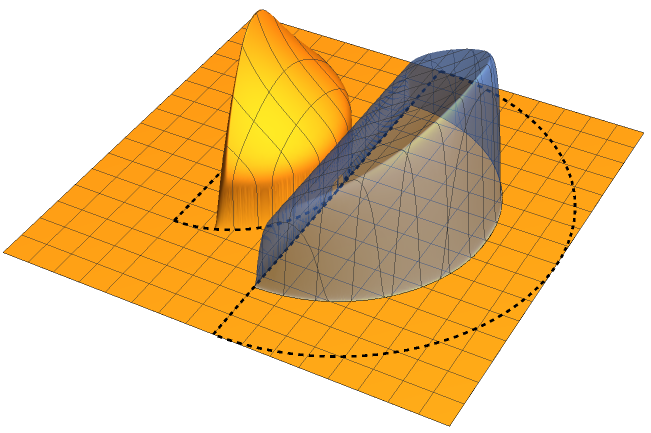}
    \captionsetup{width=0.8\textwidth}
    \caption{\small The Kelvin transform can be used to obtain a new barrier. In the picture, the dotted line represents the boundary of the unitary circle.  The red line segment (representing a subset of a hyperplane) is sent with the Kelvin transform to the red half-circle (representing a halfsphere).}
    \label{fig:kelvin}
\end{figure}

With this barrier one has the main ingredient to characterize the optimal boundary regularity. The other important ingredient is the comparison principle.  Although the operator $L_\Delta$ does not satisfy the maximum principle in general, in \cite[Corollary 1.9]{CW19} (see Theorem~\ref{mp:thm} in Section~\ref{ob:sec}) it is shown that it holds in sufficiently small domains. With this and with suitable scalings (see Section~\ref{scal:sec}) one can adapt the method of barrier functions.

As a byproduct of our approach, we can show the following Hopf-type lemma for the logarithmic Laplacian (following the ideas in \cite{JF}). Let  $\cE_L$ be the bilinear form associated to $L_\Delta$ (see~\eqref{Hdef}) and let 
\begin{align}\label{Vspace:def}
{\mathcal V}(\Omega):=\left\{ u\in L^1_{loc}(\R^N)\::\:
\int_{\R^N}\frac{|u(x)|}{(1+|x|)^{N}}\, dx+\int\int_{\substack{{x,y\in\Omega}\\{|x-y|<1}}}\frac{|u(x)-u(y)|^2}{|x-y|^N}\ dx\, dy+\int_\Omega u^2\, dx<\infty
\right\},
\end{align}
which is introduced in \cite{CW19} for the validity of a maximum principle, see Theorem~\ref{mp:thm} below.  

We say that an open set $\Omega\subset \R^N$ satisfies an \emph{interior sphere condition} at $x_0\in\partial \Omega$ if there is $\delta>0$ and $y_0\in \R^N$ with $B_{\delta}(y_0)\subset \Omega$ and $S_{\delta}(y_0) \cap \partial \Omega = \{ x_0 \}$. 

\begin{theorem}[A Hopf-type lemma for the logarithmic Laplacian]\label{strong:mp}
Let $\Omega\subset \R^N$ be an open set and assume that $\Omega$ satisfies the interior sphere condition at $x_0\in \partial \Omega$.  If $v\in{C}(\mathbb{R}^{N})\cap {\mathcal V}(\Omega)$ is a nontrivial nonnegative function such that $v(x_0)=0$ and 
\begin{align*}
\cE_L(v, \varphi) \geq 0\qquad \text{ for all $\varphi \in C^\infty_c(\Omega)$ with $\varphi \geq 0$ in $\Omega$},
\end{align*}
then
\begin{align}\label{hopf:ineq}
\liminf_{t\to 0^+} \frac{v(x_0-t\eta(x_0))}{\ell^\frac{1}{2}(t)}>0,
\end{align}
where $\eta$ is an outer unit normal vector field along $\partial \Omega$.
\end{theorem}

As a further application of our regularity results and of Theorem~\ref{strong:mp}, we also show the \emph{uniqueness of positive solutions} for the nonlinear problem 
\begin{align}\label{sub:intro}
L_{\Delta}v=-\mu\ln\left(|v|\right)v\quad\mbox{in}~\Omega,\qquad v=0\quad \text{ in }\R^N\backslash \Omega,
\end{align}
where $\mu>0$ and $\Omega$ is a bounded open set of class $C^2$, see Theorem~\ref{exisleast:intro} in Section~\ref{sublin:sec}.  The proof of this result relies on a convexity-by-paths argument, following the approach in \cite{BFMST18}. Problem~\eqref{sub:intro} appears naturally in the small-order limit of solutions to the fractional Lane-Emden equation; to be more precise, if $u_s$ is a positive solution of 
\begin{align*}
(-\Delta)^s u_s = u_s^{p_s-1}\quad \text{ in }\Omega,\qquad u_s=0\quad \text{ in }\R^N\backslash \Omega,
\end{align*}
where $p_s\in(1,2)$ is such that $\mu=\lim_{s\to 0^+}\frac{2-p_s}{s}$, then $u_s\to v$ in $L^q(\Omega)$ for $1\leq q<\infty$ as $s\to 0^+$, where $v$ is a positive solution of \eqref{sub:intro}. See \cite{AS22} for more details.  Note that these problems are \emph{sublinear}. In the superlinear case ($p_s>2$) uniqueness is known to be false in general (see \cite{davila2017bubbling,DIS23,DIS24,FW24,FW242,DP24} for uniqueness and multiplicity results for positive solutions of fractional problems). Equation~\eqref{sub:intro} with $\mu<0$ (which would be the superlinear case in the logarithmic setting) has been studied in \cite{hernandez2022small}, but nothing is known yet about the uniqueness or multiplicity of positive solutions and our techniques cannot be used in this setting. 

To close this introduction, we mention that we believe a similar strategy can also be used to study the sharp boundary behavior of more general integral operators associated with zero-order kernels, as those considered in \cite{CS22,CD18,feulefack2021nonlocal,feulefack2023fractional}, for instance.

\medskip

The paper is organized as follows. In Section~\ref{ini:sec} we construct the initial barriers described in Figure~\ref{fig1}. Section~\ref{sec:kelvin} contains the proof of Proposition~\ref{K:lem} regarding the Kelvin transform and the construction of the barrier portrayed in Figure~\ref{fig:kelvin}. Theorem~\ref{convex:thm} is shown in Section~\ref{ob:sec}, whereas
the proofs of Theorem~\ref{torsion:ball} (optimal regularity for the torsion problem), Theorem~\ref{strong:mp} (Hopf-type lemma for the logarithmic Laplacian), and the uniqueness of positive solutions of~\eqref{sub:intro} (Theorem~\ref{u:thm}) are contained in Section~\ref{app:sec}.  Finally, we include an Appendix with some useful results regarding scalings and a Leibniz-type formula for the logarithmic Laplacian of a product. 

\subsection*{Notation}

We use $B_r(x)$ to denote the open ball of radius $r>0$ centered at $x\in \R^N$. If $x=0$, then we simply write $B_r:=B_r(0).$  We also set $S_r(x)=\partial B_r(x)$, $S_r:=S_r(0),$ and 
\begin{align*}
    \sigma_N:=|S_1|=\frac{2\pi^\frac{N}{2}}{\Gamma(\frac{N}{2})}
\end{align*}
to denote the surface measure of the unit sphere.  For $U\subset \R^N,$ we use $U^c:=\R^N\backslash U.$

The constants involved in the definition of $L_\Delta$ are given by
\begin{align}\label{rhon}
c_{N}:=\pi^{-\tfrac{N}{2}}\Gamma(\tfrac{N}{2}),\qquad \rho_{N}:=2\ln 2+\Psi(\tfrac{N}{2})-\gamma,\quad \text{ and }\quad\gamma:=-\Gamma^{\prime}(1).
\end{align}
Here $\gamma$ is also known as the Euler-Mascheroni constant and $\Psi:=\frac{\Gamma^{\prime}}{\Gamma}$ is the digamma function.

Next, following \cite{CW19}, we introduce the variational framework for $L_\Delta$. Let $\Omega \subset \mathbb{R}^N$ be an open bounded set and let $\mathbb{H}(\Omega)$ be the Hilbert space given by
\begin{align}\label{Hdef}
\mathbb{H}(\Omega):=\left\lbrace u\in L^{2}(\mathbb{R}^{N})~:~\int\int_{\substack{x,y\in\mathbb{R}^{N}\\ |x-y|\leq 1}}\frac{|u(x)-u(y)|^{2}}{|x-y|^{N}}\, dx\, dy<\infty~\mbox{and}~u=0~\mbox{in}~\mathbb{R}^{N}\setminus\Omega\right\rbrace
\end{align}
with the inner product 
\begin{align*}
\mathcal{E}(u,v):=\frac{c_{N}}{2}\int_{\R^N}\int_{B_1(x)}\frac{(u(x)-u(y))(v(x)-v(y))}{|x-y|^{N}}\, dy\, dx,
\end{align*}
and the norm
\begin{align}\label{norm:def}
\|u\|:=\left(\mathcal{E}(u,u)\right)^{\tfrac{1}{2}}.    
\end{align}
The operator $L_{\Delta}$ has the following associated quadratic form 
\begin{align}\label{eq:bilog}
\mathcal{E}_{L}(u,v)
&:=\mathcal{E}(u,v)-c_{N}\int\int_{\substack{x,y\in\mathbb{R}^{N}\\ |x-y|\geq 1}}\frac{u(x)v(y)}{|x-y|^{N}}\, dx\, dy+\rho_{N}\int_{\mathbb{R}^{N}}uv\, dx.
\end{align}
Furthermore, for $u\in\mathbb H(\Omega)$,
\begin{align*}
\mathcal{E}_{L}(u,u)=\frac{c_{N}}{2}\int_\Omega\int_\Omega \frac{(u(x)-u(y))^2}{|x-y|^{N}}\, dx\, dy
+\int_\Omega (h_\Omega(x)+\rho_N)u(x)^2\, dx,
\end{align*}
where 
\begin{equation}
\label{eq:h_omega}
h_\Omega(x):=c_N\left(\int_{B_1(x)\backslash \Omega}|x-y|^{-N}\, dy-\int_{ \Omega\backslash B_1(x)}|x-y|^{-N}\, dy\right),  
\end{equation}
see \cite[Proposition 3.2]{CW19}.  

For $f\in L^2(\Omega)$, we say that $u\in \mathbb H(\Omega)$ is a weak solution of $L_\Delta u = f$ in $\Omega$, $u=0$ in $\R^N\backslash \Omega$, if 
\begin{align*}
    \cE_L(u,\varphi) = \int_\Omega f \varphi\, dx\qquad \text{ for all }\varphi \in C^\infty_c(\Omega).
\end{align*}
By \cite[Theorem 1.1]{CW19}, it holds that 
\begin{align*}
\mathcal{E}_{L}(u,u)=\int_{\mathbb{R}^{N}}\ln(|\xi|^2)|\hat{u}(\xi)|^{2}\, d\xi\qquad\mbox{for all}~u\in{C}_{c}^{\infty}(\Omega),
\end{align*}
where $\hat u$ is the Fourier transform of $u$. Moreover, for $\varphi\in{C}_{c}^{\infty}(\Omega),$ we have that $L_{\Delta}\varphi\in L^{p}(\mathbb{R}^{N})$ and 
\begin{align}\label{eq:725}
\mathcal{E}_{L}(u,\varphi)=\int_{\Omega}uL_{\Delta}\varphi \, dx
 \qquad \text{ for $u\in\mathbb{H}(\Omega)$,}
\end{align}
see \cite[Theorem 1.1]{CW19}. 

Let $v: \Omega \rightarrow \mathbb{R}$ be a measurable function. The modulus of continuity of $v$ at a point $x \in \Omega$ is defined by
$$
\omega_{v, x, \Omega}:(0, \infty) \rightarrow[0, \infty), \quad \omega_{v, x, \Omega}(r)=\sup _{\substack{y \in \Omega \\|y-x| \leq r}}|v(y)-v(x)| .
$$
The function $v$ is called \emph{Dini continuous} at $x$ if $\int_0^1 \frac{\omega_{v, x, \Omega}(r)}{r} d r<\infty$. If
$$
\int_0^1 \frac{\omega_{v, \Omega}(r)}{r} d r<\infty \quad \text { for the uniform continuity modulus } \omega_{v, \Omega}(r):=\sup _{x \in \Omega} \omega_{v, x, \Omega}(r),
$$
then we call $v$ \emph{uniformly Dini continuous} in $\Omega$.

Let $\ell:[0,\infty)\to[0,\infty)$ be given by
\begin{align*}
\ell(r)=-\frac{1}{\ln(\min\{0.1,r\})}.
\end{align*}
\setlength{\unitlength}{1cm}
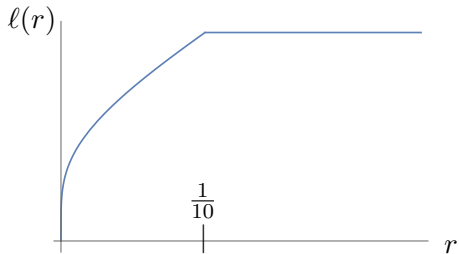
\begin{figure}[ht] 
\begin{center}
\begin{picture}(5,3.5)
\put(0,0){\includegraphics[width=5cm]{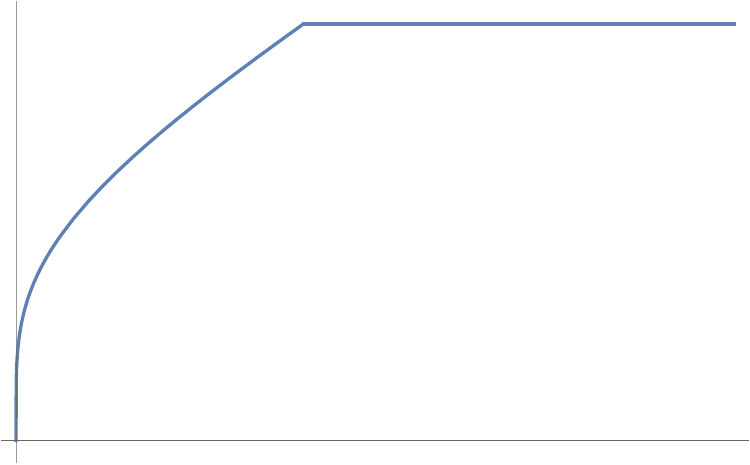}}
\put(-0.6,3){$\ell(r)$}
\put(5.2,0){$r$}
\put(1.8,0.6){$\frac{1}{10}$}
\put(1.95,0.1){$|$}
\end{picture}
\end{center}
 \caption{The function $\ell$.}\label{f}
\end{figure}

For $\alpha>0$, we also define the $\alpha$-log-Hölder Banach space (see \cite[Lemma 7.1]{CS22}) by 
\begin{align}\label{log:H:def}
\mathcal L^\alpha(\Omega):=\{u:\Omega\to\R\::\: \|u\|_{\mathcal L^\alpha(\Omega)}<\infty\},
\end{align}
where 
\begin{align*}
\|u\|_{\mathcal L^\alpha(\Omega)} := \|u\|_{L^\infty(\Omega)} + [u]_{\mathcal L^{\alpha}(\Omega)},\qquad 
[u]_{\mathcal L^\alpha(\Omega)} := \sup_{\substack{x,y\in \Omega\\x\neq y}}\frac{|u(x)-u(y)|}{\ell^\alpha(|x-y|)},
\end{align*}
and $\|\cdot\|_{L^\infty(\Omega)}$ is the usual norm in $L^\infty(\Omega)$. We also set $\|u\|_\infty:=\|u\|_{L^\infty(\R^N)}$.

 We shall use the following semi-homogeneity property for the modulus of continuity $\ell$.  
\begin{lemma}[Lemma 3.2 in \cite{CS22}]\label{prop1}
There is $c>0$ such that 
\begin{align*}
\ell(\lambda r)\geq c\, \ell(\lambda)\ell(r)\,\qquad \text{ for all }r,\lambda >0.    
\end{align*}
\end{lemma}

\section{An initial barrier}\label{ini:sec}

The main goal of this section is to show the following result. Let $N \geq 1$ and let 
\begin{align}\label{Omega:def}
    B_2^+:=B_2\cap \R^N_+, \qquad  \text{ where }\R^N_+:=\{x\in\R^N\::\: x_1>0\}.
\end{align}

\begin{theorem}\label{barrier:thm}
There is a function $w\in \cL^{\frac{1}{2}}(\R^N)\cap C^\infty(\R^N\backslash \R^N_+)\cap {\mathbb H}(B_2^+)$ such that $w=0$ in $\R^N\backslash B_2^+$ and
\begin{align*}
L_\Delta w\in L^\infty(B_2^+).
\end{align*}
Moreover, there is $c>0$ and $\delta>0$ such that
\begin{align*}
c\,\ell^{\frac{1}{2}}(x_1) <w(x)<c^{-1}\ell^{\frac{1}{2}}(x_1)
\quad \text{ and }\quad 
L_\Delta w(x)>c\qquad \text{ for all }x\in B_2^+\cap B_\delta.
\end{align*}
\end{theorem}

To show Theorem~\ref{barrier:thm} we construct explicitly the function $w$ and estimate its logarithmic Laplacian with direct calculations. 

We show first some auxiliary lemmas. Let $\Omega\subset\R^N$ be an open bounded set. 

\begin{lemma}\label{Ldelta}
Let $V\in L^\infty(\R^N)$ be such that $V=0$ in $\R^N\backslash \Omega$ and
\begin{align}\label{C:bd}
\sup_{z\in\Omega}\left|\int_{B_1}
\frac{
V(z)
-V(z+y)}{|y|^N}\, dy\right|<C
\end{align}
for some $C>0$. Then $L_\Delta V\in L^\infty(\Omega)$. 
\end{lemma}
\begin{proof}
Let $x\in \Omega$, then
\begin{align*}
|L_\Delta V(x)|
&=\left|c_N\int_{B_1}\frac{V(x)-V(x+y)}{|y|^N}\, dy-c_N\int_{\Omega\backslash B_1(x)}\frac{V(y)}{|y-x|^N}\, dy+\rho_N V(x)\right|\\
&\leq c_NC+\|V\|_\infty(c_N|\Omega|+|\rho_N|)<\infty.
\end{align*}
\end{proof}

\begin{lemma}\label{V:lem:1}
Let $V\in C(\overline{\Omega})$ be such that $V=0$ on $\partial\Omega$ and $\eta\in[\frac{3}{4},1]$. Let $U\subset \Omega$ be an open subset such that $\overline{U}\cap\partial\Omega\neq \emptyset$ and
\begin{align}\label{d}
\inf_{z\in U}
\int_{B_1}
\frac{
V(z)
-V(z+y)}{|y|^N}\, dy\geq (1+\eta)\sigma,\qquad 
\sup_{z\in U}
c_N\int_{\R^N\backslash B_1}
\frac{
|V(z+y)|}{|y|^N}\, dy\leq \frac{\sigma}{2}
\end{align}
for some $\sigma>0$. Then there is an open subset $U'\subset U$ such that $\overline{U'}\cap\partial\Omega=\overline{U}\cap\partial\Omega$ and 
\begin{align*}
L_\Delta V\geq \sigma\qquad \text{ in $U'$.}
\end{align*}
\end{lemma}
\begin{proof}
Since $V=0$ on $\partial \Omega$ and $\overline{U}\cap\partial\Omega\neq \emptyset$, there is, by continuity, an open subset $U'\subset U$ with $\overline{U'}\cap\partial\Omega=\overline{U}\cap\partial\Omega$ and such that $|\rho_N V|<\frac{\sigma}{4}$ in $U'$. Then, by~\eqref{d}, for every $x\in U'$,
\begin{align*}
L_\Delta V(x)
&=c_N\int_{B_1}\frac{V(x)-V(x+y)}{|y|^N}\, dy-c_N\int_{\R^N\backslash B_1}\frac{V(x+y)}{|y|^N}\, dy+\rho_N V(x)\\
&\geq (1+\eta)\sigma-\frac{\sigma}{2}-\frac{\sigma}{4}\geq \sigma.
\end{align*}
\end{proof}

Now, we split the proof of Theorem~\ref{barrier:thm} in two cases: the (simpler) one-dimensional case and the higher-dimensional case. 

\subsection{The one-dimensional case}

Let 
\begin{align}\label{zeta}
0<\zeta<\frac{\sqrt{\ln 2}}{4}\sqrt{\ln \frac{4}{3}}<\frac{1}{4},
\end{align}
let $\varphi\in C^\infty_c(\R)$ be an even function such that
\begin{align} \label{phi:def}
\varphi=1\quad\text{ in $(-1-\zeta,1+\zeta)$},\qquad 
\varphi=0\quad \text{ in $\R\backslash (-1-2\zeta,1+2\zeta)$,}\qquad
0\leq \varphi \leq 1\quad \text{ in $\R$},
\end{align}
and let $u:\R\to\R$ be given by
 \begin{align}\label{u}
     u(x):=\frac{1}{\sqrt{-\ln \frac{x}{2}}}\varphi(x)\chi_{[0,\infty)}(x),
 \end{align}
 where $\chi_{[0,\infty)}$ is the characteristic function of $[0,\infty)$, see Figure~\ref{fig1} (left).

The goal of this section is to show Theorem~\ref{barrier:thm} in the case $N=1$, which we state next.
 
 \begin{theorem}\label{thm:L} The function $u$ given by~\eqref{u} belongs to ${\mathbb H}((0,2))\cap C^\infty(0,\infty)\cap \cL^{\frac{1}{2}}(\R)$ and there is $\delta\in(0,2)$ such that
\begin{align*}
L_\Delta u\in L^\infty((0,2))\qquad \text{ and }\qquad L_\Delta u \geq \sqrt{\ln 2}\quad \text { in }(0,\delta).
\end{align*}
\end{theorem}

We show first some auxiliary lemmas.

\begin{lemma}\label{r:lem}
The function $u$ given by~\eqref{u} satisfies that 
\begin{align*}
\int_{\R\backslash (-1,1)} \frac{u(\eps+y)}{|y|}\ dy <\frac{1}{2}\sqrt{\ln 2}\qquad \text{ for }\eps\in(0,2\zeta).
\end{align*}
\end{lemma}
\begin{proof}
By~\eqref{zeta} it follows that, for $\eps\in (0,2\zeta)\subset (0,\frac{1}{2})$,
\begin{align*}
0<\int_{\R\backslash (-1,1)} \frac{u(\eps+y)}{|y|}\ dy
&=\int_{1+\eps}^{1+2\zeta} \frac{1}{\sqrt{-\ln \frac{y}{2}}}\frac{\varphi(y)}{|y-\eps|}\ dy
\leq  \frac{1}{\sqrt{-\ln \frac{1+2\zeta}{2}}}\int_{1}^{1+2\zeta-\eps}\frac{1}{|y|}\ dy\\
&\leq\frac{2\zeta}{\sqrt{-\ln \frac{1+\frac{1}{2}}{2}}}=\frac{2\zeta}{\sqrt{-\ln \frac{3}{4}}}<\frac{1}{2}\sqrt{\ln 2}.
\end{align*}
\end{proof}

\begin{lemma}\label{J1:s}
It holds that
\begin{align*}
    \int_{-\eps}^{\eps} \frac{|u(\eps)-u(y+\eps)|}{|y|}\ dy=o(1)\qquad \text{ as }\eps\to 0.
\end{align*}
In particular,
\begin{align*}
J_1=J_1(\eps):=\int_{0}^{2\eps} \frac{u(\eps)-u(y)}{|\eps-y|}\ dy=o(1)\qquad \text{ as }\eps\to 0.
\end{align*}
\end{lemma}
\begin{proof}
Let $\eps>0$ be small enough so that $u$ is increasing in $(0,\eps)$.  If $y<0$, then $(y+1)\eps<\eps$ and $u(\eps)>u(\eps)-u((y+1)\eps)>0$. Therefore,
\begin{align*}
\int_{-1}^{-\frac{1}{2}}
\frac{|u(\eps)-u((y+1)\eps)|}{|y|}
\ dy
\leq 
2\int_{-1}^{-\frac{1}{2}}
u(\eps)
\ dy
=u(\eps)
=o(1)\qquad \text{ as }\eps\to 0,
\end{align*}
because $u(0)=0$. Using that $t\mapsto |u'(t)|$ is decreasing in $(0,\frac{1}{4})$, it follows that 
\begin{align*}
\frac{\eps}{2}<\eps((1-s)y+1)<2\eps<\frac{1}{4}    \qquad \text{ for }s\in(0,1)\text{ and }y\in(-\tfrac{1}{2},1),
\end{align*}
and then
\begin{align*}
&\int_{-\frac{1}{2}}^1\frac{|u(\eps)-u((y+1)\eps)|}{|y|}\, dy
\leq\eps\int_{-\frac{1}{2}}^1 \int_0^1 |u'(\eps((1-s)y+1))| \,ds\, dy\leq \frac{3}{2}\eps |u'(\tfrac{1}{4}\eps)|
=o(1)
\end{align*}
as $\eps\to 0$.
\end{proof}

\begin{lemma}\label{J2:s:aux}
It holds that
\begin{align*}
 I:=I(\eps):=\int_{\eps}^{1}\frac{1}{y}
 \left(
 \frac{1}{\sqrt{-\ln{(y+\eps)+\ln 2}}}-\frac{1}{\sqrt{-\ln (y)+\ln 2}}
 \right)\ dy=o(1)\qquad \text{ as }\eps\to 0.
\end{align*}
\end{lemma}
\begin{proof}
Using the change of variables $z=-\ln y$ ($y=e^{-z}$, $dy=-e^{-z}dz$) we obtain that
\begin{align*}
I&=\int_{0}^{-\ln\eps}\frac{1}{\sqrt{-\ln(e^{-z}+\eps)+\ln 2}}
-
\frac{1}{\sqrt{z+\ln 2}}\ dz=o(1)
\end{align*}
as $\eps\to 0$. Indeed, consider
\begin{align*}
 F_\eps(z) := \left(\frac{1}{\sqrt{-\ln(e^{-z}+\eps)+\ln 2}}
-
\frac{1}{\sqrt{z+\ln 2}}\right)
\chi_{\{z<-\ln\eps\}}(z),
\end{align*}
for $z\in(0,\infty)$ and $\eps \in (0,1)$. Observe that
\begin{equation*}
    \frac{1}{\sqrt{s}} - \frac{1}{\sqrt{t}} \le \frac{1}{s^{3/2}} (t-s), \quad \ln{t} - \ln{s} \le \frac{1}{s}(t-s) \quad \text{for all } 0<s \le t,
\end{equation*}
by the mean value theorem. By applying these estimations to $F_\eps$ we deduce that
\begin{equation} \label{eq:dominated}
\begin{aligned}   
    F_\eps(z) &\le \frac{z+\ln(e^{-z}+\eps)}{2\left[-\ln(e^{-z} +\eps) + \ln 2 \right]^{3/2}} = \frac{\ln(e^{-z}+\eps) - \ln \left( e^{-z} \right)}{2\left[-\ln(e^{-z} +\eps) + \ln 2 \right]^{3/2}} \\
     & \le \frac{e^z \eps}{2\left[-\ln(e^{-z} +\eps) + \ln 2 \right]^{3/2}} \le \frac{1}{2\left[-\ln(e^{-z} +\eps) + \ln 2 \right]^{3/2}} \qquad \text{for } z \le -\ln \eps. 
\end{aligned}
\end{equation}
So let us define, for any $\eps \ge 0$,
\begin{equation*}
    G_\eps(z) := \frac{1}{2\left[-\ln(e^{-z} +\eps) + \ln 2 \right]^{3/2}}, \quad z \in \R.
\end{equation*}
Observe that $G_\eps(z) \searrow G_0(z)$ as $\eps \to 0$ for $z\in\R$, so the monotone convergence theorem implies that $\int_0^\infty G_\eps(z) \searrow \int_0^\infty G_0(z)$ as $\eps \to 0$. On the other hand, by~\eqref{eq:dominated},  $F_\eps(z) \le G_\eps(z)$ for $z \in \R$ and $F_\eps(z)\to F_0(z)=0$ as $\eps\to 0$ for $z\in\R$, so the generalized Lebesgue dominated convergence theorem implies that $\int_0^\infty F_\eps \to 0$ as $\eps \to 0$, as claimed.
\end{proof}

\begin{lemma}\label{J2:s}
It holds that
\begin{align}\label{J2:c:1}
J_2=J_2(\eps):=-\int_{2\eps}^{1+\eps}\frac{u(y)}{|\eps-y|}\ dy = 2\sqrt{\ln 2}-2\sqrt{-\ln \eps}+o(1)\qquad \text{ as }\eps\to 0.
\end{align}
\end{lemma}
\begin{proof}
Let $\eps\in(0,\zeta)$ with $\zeta$ as in~\eqref{zeta}. Observe that
\begin{align}\label{a1}
\int_{\eps}^1 \frac{1}{y\sqrt{-\ln \frac{y}{2}}}\ dy=2\sqrt{-\ln\frac{\eps}{2}}-2\sqrt{\ln 2}.
\end{align}
Then, using~\eqref{a1},
\begin{align}
J_2
&=-\int_{2\eps}^{1+\eps}\frac{u(y)}{|\eps-y|}\ dy 
=-\int_{\eps}^{1}\frac{u(y+\eps)}{y}\ dy\notag\\
&=2\sqrt{\ln 2}-2\sqrt{-\ln\frac{\eps}{2}}
+\int_{\eps}^{1}\frac{1}{y}\left(\frac{1}{\sqrt{-\ln \frac{y}{2}}}-u(y+\eps)\right)\ dy\notag\\
&=
2\sqrt{\ln 2}-2\sqrt{-\ln\frac{\eps}{2}}
+\int_{\eps}^{1}\frac{1}{y}\left(\frac{1}{\sqrt{-\ln (y)+\ln 2}}-\frac{1}{\sqrt{-\ln{(y+\eps)+\ln 2}}}\right)\ dy\notag\\
&=
2\sqrt{\ln 2}-2\sqrt{-\ln\frac{\eps}{2}}
-I,\label{J2:1}
\end{align}
with $I$ as in Lemma~\ref{J2:s:aux}. Therefore, by~\eqref{J2:1} and Lemma~\ref{J2:s:aux},
\begin{align*}
J_2=2\sqrt{\ln 2}-2\sqrt{-\ln \eps}+o(1)\qquad \text{ as }\eps\to 0.
\end{align*}
\end{proof}

\begin{theorem}\label{thm:bdry:s} The function $u$ given by~\eqref{u} belongs to $C^\infty((0,\infty))$ and
\begin{align}\label{cl4}
J&=J(\eps):=\int^{\eps+1}_{\eps-1} \frac{u(\eps)-u(y)}{|y-\eps|}\ dy = 2\sqrt{\ln 2}+o(1)\qquad \text{ as }\eps\to 0^+.
\end{align}
\end{theorem}
\begin{proof}
That $u\in C^\infty((0,\infty))$ is clear. Moreover, since $u=0$ in $(-\infty,0)$,
\begin{align}
J&=\int^{\eps+1}_{0} \frac{u(\eps)-u(y)}{|\eps-y|}\ dy
+u(\eps)\int^{0}_{\eps-1}|\eps-y|^{-1}\ dy=\int^{\eps+1}_{0} \frac{u(\eps)-u(y)}{|\eps-y|}\ dy
+u(\eps)(-\ln \eps)\label{e1:s}
\end{align}
for $\eps>0$ small.  Let $J_1$ and $J_2$ as in Lemmas~\ref{J1:s} and~\ref{J2:s}. Then,
\begin{align}
\int_{0}^{1+\eps} \frac{u(\eps)-u(y)}{|\eps-y|}\ dy
&=
\int_{0}^{2\eps} \frac{u(\eps)-u(y)}{|\eps-y|}\ dy
+
\int_{2\eps}^{1+\eps} \frac{u(\eps)-u(y)}{|\eps-y|}\ dy\notag\\
&=J_1+J_2+u(\eps)\int_{2\eps}^{1+\eps} \frac{1}{|\eps-y|}\ dy
=J_1+J_2+u(\eps)(-\ln \eps).\label{e2:s}
\end{align}
Now the claim~\eqref{cl4} follows by~\eqref{e1:s},~\eqref{e2:s}, Lemma~\ref{J1:s}, Lemma~\ref{J2:s}, and the fact that
\begin{align*}
2u(\eps)(-\ln(\eps))=2\sqrt{-\ln \eps}+o(1)\qquad \text{ as $\eps\to 0^+$.}
\end{align*}
\end{proof}

 \begin{lemma}\label{uinH:lem}
The function $u$ given by~\eqref{u} belongs to ${\mathbb H}((0,2))$.
 \end{lemma}
 \begin{proof}
Let $\zeta$ be as in~\eqref{zeta}. For $x\in U:=(-3,3)\backslash(-\zeta,\zeta)$ and $y\in(-1,1),$ using a Taylor expansion, we have that
\begin{align*}
    |u(x+y)-u(x)|=|yu'(x)+|y|h(y)|\leq |y|(\|u'\|_{L^\infty(U)}+|h(y)|)
\end{align*}
for some $h\in L^\infty((-1,1))$ such that $h(y)=o(1)$ as $y\to 0.$  Then,
\begin{align*}
A_1&:=\int_{U}\int_{-1}^{1}\frac{(u(y+x)-u(x))u(y+x)}{|y|} \, dy\, dx
\leq \|u\|_\infty \int_{U}\int_{-1}^{1} \|u'\|_{L^\infty(U)}+\|h\|_{L^\infty((-1,1))}\, dy\, dx<\infty,\\
A_2&:=\int_{0}^{\zeta}\int_{x}^{1}\frac{(u(y+x)-u(x))u(y+x)}{|y|} \, dy\, dx
    \leq 2\|u\|^2_\infty \int_{0}^{\zeta}-\ln(x)\, dx<\infty.
\end{align*}
Moreover, using that $u=0$ in $(-\infty,0)$ and Lemma~\ref{J1:s},
\begin{align*}
A_3&:=\int_{0}^{\zeta}\int_{-1}^{x}\frac{(u(y+x)-u(x))u(y+x)}{|y|} \, dy\, dx
\leq \|u\|_\infty\int_{0}^{\zeta}\int_{-x}^{x}\frac{|u(x)-u(y+x)|}{|y|} \, dy\, dx<\infty.
\end{align*}
Furthermore,  using that $u=0$ in $(-\infty,0)$,
\begin{align*}
A_4&
:=\int_{-\zeta}^{0}\int_{-1}^{1}\frac{(u(y+x)-u(x))u(y+x)}{|y|} \, dy\, dx
=\int_{-\zeta}^{0}\int_{-x}^{1}\frac{(u(y+x)-u(x))u(y+x)}{|y|} \, dy\, dx\\
&\leq 2\|u\|^2_\infty\int_{-\zeta}^{0}-\ln(-x)\, dx<\infty.
\end{align*}
Then,
\begin{align}\label{As}
A:=\int_{\R}\int_{-1}^{1}\frac{(u(y+x)-u(x))u(y+x)}{|y|}\, dy\, dx=A_1+A_2+A_3+A_4<\infty.
\end{align}

Recall the definition of the norm $\|\cdot\|$ given in~\eqref{norm:def}.  Then,
\begin{align*}
2\|u\|^2
&=\int_{\R}\int_{x-1}^{x+1}\frac{(u(x)-u(y))(u(x)-u(y))}{|x-y|}\, dy\, dx\\
&=\int_{\R}\int_{x-1}^{x+1}
    \frac{(u(x)-u(y))u(x)}{|x-y|}\, dy\, dx
    + \int_{\R}\int_{x-1}^{x+1}\frac{(u(y)-u(x))u(y)}{|x-y|}
    \, dy\, dx.
\end{align*}
Using Fubini's theorem, a change of variables ($z=x-y$), and~\eqref{As},
\begin{align*}
\int_{\R} \int_{x-1}^{x+1}
    \frac{(u(x)-u(y))u(x)}{|x-y|}\, dy\, dx
&=
\int_{\R} \int_{y-1}^{y+1}
    \frac{(u(x)-u(y))u(x)}{|x-y|}\, dx\, dy\\
    &=
\int_{\R} \int_{-1}^{1}
    \frac{(u(z+y)-u(y))u(z+y)}{|z|}\, dz\, dy=A<\infty.
\end{align*}
Similarly,
\begin{align*}
&\int_{\R}\int_{x-1}^{x+1}\frac{(u(y)-u(x))u(y)}{|x-y|}\, dy\, dx
=\int_{\R}\int_{-1}^{1}\frac{(u(z+x)-u(x))u(z+x)}{|z|}\, dz\, dx=A<\infty.
\end{align*}
Thus $\|u\|<\infty$.  Since we also know that $u\in L^2(\R)$, it follows that $u\in {\mathbb H}((0,2))$ as claimed.
\end{proof}

We are ready to show Theorem~\ref{thm:L}. 

\begin{proof}[Proof of Theorem~\ref{thm:L}]
By Lemma~\ref{uinH:lem} and by the definition of $u$ it holds that $u\in{\mathbb H}((0,2))\cap C^\infty((0,\infty))\cap \cL^{\frac{1}{2}}(\R^N)$. By Theorem~\ref{thm:bdry:s} and the fact that $u\in C^\infty((0,\infty))$ we have that~\eqref{C:bd} holds (with $\Omega=(0,2)$ and $V=u$). Then, by Lemma~\ref{Ldelta}, we have that $L_\Delta u\in L^\infty((0,2))$.  Finally, by Lemma~\ref{V:lem:1} (with $\sigma=\sqrt{\ln 2}$), Lemma~\ref{r:lem}, and Theorem~\ref{thm:bdry:s} we have that $L_\Delta u\geq \sqrt{\ln 2}$ in $(0,\delta)$ for some $\delta>0$.
\end{proof}

\subsection{The higher-dimensional case}

In this section we extend the ideas of the one-dimensional case to higher dimensions.  Let $N\geq 2$,
\begin{align*}
\R^N_+:=\{x\in\R^N\::\: x_1>0\},    
\end{align*}
and let $\zeta$ be such that
\begin{align}\label{zeta2}
\zeta<\frac{1}{2}\left(\left(1+\frac{N\sqrt{\ln 2}}{4}\sqrt{\ln \frac{4}{3}}\right)^\frac{1}{N}-1\right).
\end{align}
Note that, since $(1+Nx)^{\frac{1}{N}}<1+x$ for $x\geq 0$, inequality~\eqref{zeta2} implies~\eqref{zeta}; namely, 
\begin{align*}
0<\zeta<\frac{\sqrt{\ln 2}}{4}\sqrt{\ln \frac{4}{3}}<\frac{1}{4}.
\end{align*}

Let $u$ and $\varphi$ as in~\eqref{phi:def} and~\eqref{u}; namely, $\varphi\in C^\infty_c(\R)$ is an even function such that
\begin{align*}
\varphi=1\quad\text{ in $(-1-\zeta,1+\zeta)$},\qquad 
\varphi=0\quad \text{ in $\R\backslash (-1-2\zeta,1+2\zeta)$,}\qquad
0\leq \varphi \leq 1\quad \text{ in $\R$}
\end{align*}
and $u(x):=(-\ln \frac{x}{2})^{-\frac{1}{2}}\varphi(x)\chi_{[0,\infty)}(x)$ for $x\in\R.$  Let
\begin{align}\label{V}
V \in L^1_{loc}(\R^N)\quad \text{ be given by }\quad V(x):=u(x_1)=\frac{\varphi(x_1)}{\sqrt{-\ln \frac{x_1}{2}}}\chi_{[0,\infty)}(x_1),
\end{align}
see Figure~\ref{fig1} (right).  For $\eps\in(0,\frac{1}{4})$, let
\begin{align*}
x_\eps=(\eps,0,\ldots,0)\in\R^N\quad \text{ and }\quad U_\eps:=\{y\in B_1(0)\::\: y_1\in(-\eps,\eps)\}.
\end{align*}
For $N\geq 1$, we use $\sigma_N$ to denote the surface measure of the sphere in $\R^N$, namely, $\sigma_N=\frac{N\pi^\frac{N}{2}}{\Gamma(\frac{N}{2}+1)}$.

We show first some auxiliary lemmas.
\begin{lemma}\label{J1}
\begin{align*}
J_1=J_1(\eps):=\int_{U_\eps} \frac{V(x_\eps)-V(y+x_\eps)}{|y|^N}\ dy=o(1)\qquad \text{ as }\eps\to 0.
\end{align*}
\end{lemma}
\begin{proof}
Let $B:=\{y'\in \R^{N-1}\::\: |y'|<1\}$. Using that $u$ is increasing in $(0,\eps)$ for $\eps>0$ small and that $u(0)=0$,
\begin{align*}
J_{1,1}&:=
\int_{B}\int_{-\eps}^{-\frac{\eps}{2}}\frac{|u(\eps)-u(y_1+\eps)|}{(y_1^2+|y'|^2)^{\frac{N}{2}}}\ dy_1\,dy'
\leq u(\eps)\int_{B} \frac{\eps}{2}\left(\frac{\eps^2}{4}+|y'|^2\right)^{-\frac{N}{2}}\,dy'\\
&=\frac{\eps}{2}u(\eps)\sigma_{N-1}\int_{0}^1 \left(\frac{\eps^2}{4}+r^2\right)^{-\frac{N}{2}}r^{N-2}\,dr
\leq u(\eps)\sigma_{N-1}\int_{0}^\infty \left(1+r^2\right)^{-\frac{N}{2}}r^{N-2}\,dr=o(1)
\end{align*}
as $\eps\to 0$. Furthermore, using that $t\mapsto|u'(t)|$ is decreasing in $(0,\frac{1}{3})$,
\begin{align*}
J_{1,2}
&:=\int_{B}\int_{-\frac{\eps}{2}}^\eps\frac{\left|u(\eps)-u(y_1+\eps)\right|}{(y_1^2+|y'|^2)^{\frac{N}{2}}}\ dy_1\,dy'
\leq \int_{B}\int_{-\frac{\eps}{2}}^\eps \int_0^1 |u'(ty_1+\eps)|\frac{y_1}{(y_1^2+|y'|^2)^{\frac{N}{2}}}\,dt\, dy_1\,dy'\\
&\leq |u'(\tfrac{\eps}{2})|\int_{B}\int_{-\frac{\eps}{2}}^\eps \frac{y_1}{(y_1^2+|y'|^2)^{\frac{N}{2}}}\, dy_1\,dy'
 =\sigma_{N-1}|u'(\tfrac{\eps}{2})|\int_{0}^1\int_{-\frac{\eps}{2}}^\eps \frac{y_1}{(y_1^2+r^2)^{\frac{N}{2}}}r^{N-2}\, dy_1\,dr\\
 &\leq \sigma_{N-1}\eps|u'(\tfrac{\eps}{2})|\int_{0}^\infty \int_{-\frac{1}{2}}^1 \frac{y_1\,r^{N-2}}{(y_1^2+r^2)^{\frac{N}{2}}}\, dy_1\,dr
 =\frac{3}{2}\sigma_{N-1}\eps|u'(\tfrac{\eps}{2})|\int_{0}^\infty \frac{r^{N-2}}{(1+r^2)^{\frac{N}{2}}}\,dr
 =o(1)
\end{align*}
as $\eps\to 0^+$, where we used that $\eps|u'(\tfrac{\eps}{2})|\to 0$ as $\eps\to 0^+$. Then,
\begin{align}\label{lem211}
\int_{U_\eps} \frac{|V(x_\eps)-V(y+x_\eps)|}{|y|^N}\ dy\leq J_{1,1}+J_{1,2}=o(1)\qquad \text{ as $\eps\to 0$}
\end{align}
 and this ends the proof.
\end{proof}

For $\eps\in(0,\frac{1}{4})$ and $\eta\in\left[\frac{1}{\sqrt{2}},1\right]$, let
\begin{align*}
Q_{\eps,\eta}&:=\{y=(y_1,y')\in \R^N\::\: y_1\in(\eps,\eta),\ |y'|<\eta\}.
\end{align*}
Note that $Q_{\eps,\frac{1}{\sqrt{2}}}\subset B_1$ and $\{y\in B_1\::\: y_1>\eps\}\subset Q_{\eps,1}$.

\begin{lemma}\label{J21:lem}
For $\eps\in(0,\frac{1}{4})$ and $\eta\in[\frac{1}{\sqrt{2}},1]$, let
\begin{align*}
J_{21}(\eta):=\int_{Q_{\eps,\eta}} \frac{V(y)}{|y|^N}\ dy>0.
\end{align*}
Then $J_{21}(\eta)=\sigma_N\sqrt{-\ln\eps}+O(1)$ and  $J_{21}(\eta)<\sigma_N\sqrt{-\ln \eps}-\sigma_N\sqrt{\ln 2-\ln \eta}+o(1)$ as $\eps\to 0$.
\end{lemma}
\begin{proof}
For $\eps\in(0,\frac{1}{4})$ and $\eta\in[\frac{1}{\sqrt{2}},1]$, let
\begin{align*}
R(\eta):=\sigma_{N-1}\int_{\frac{\eps}{\eta}}^{1}|t^2+1|^{-\frac{N}{2}} 
\int_{\ln t-\ln\eta}^{-\ln \eta} \frac{1}{\sqrt{\ln 2+z-\ln t}}\, dz\,dt>0.
\end{align*}

\setlength{\unitlength}{\textwidth}
\begin{figure}[ht]
\begin{minipage}{0.45\textwidth}
\begin{center}
\begin{picture}(.4,.2)
\put(0,0){\includegraphics[width=.3\paperwidth]{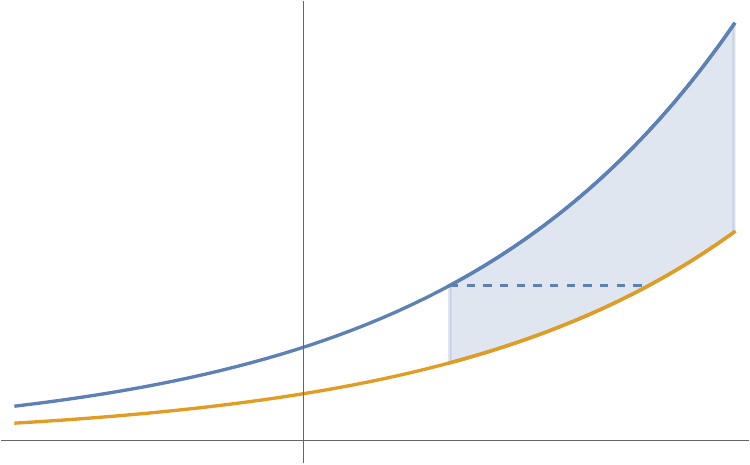}}
\put(.4,0){$z$}
\put(.2,-0.01){$-\ln \eta$}
\put(0.14,0.09){$1$}
\put(0.155,0.09){-}
\put(.232,0.008){$\shortmid$}
\put(0.145,0.23){$t$}
\put(0.28,0.18){$\eta e^z$}
\put(0.3,0.06){$\eps e^z$}
\end{picture}
\end{center}
\end{minipage}
\begin{minipage}{0.5\textwidth}
\begin{center}
\begin{picture}(.4,.25)
\put(0,0){\includegraphics[width=.3\paperwidth]{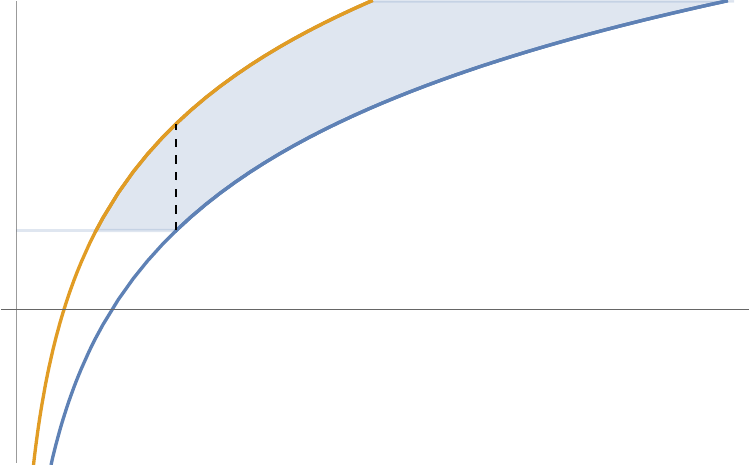}}
\put(.4,0.065){$t$}
\put(0.02,0.23){$z$}
\put(-0.06,0.12){$-\ln\eta$}
\put(0.06,0.21){$\ln t- \ln \eps$}
\put(0.29,0.18){$\ln t- \ln \eta$}
\put(0.089,0.075){$\shortmid$}
\put(0.046,0.075){$\shortmid$}
\put(0.049,0.05){$\frac{\eps}{\eta}$}
\put(0.084,0.05){$1$}
\end{picture}
\end{center}
\end{minipage}

\caption{Domains of integration.}\label{fig:2}
\end{figure}
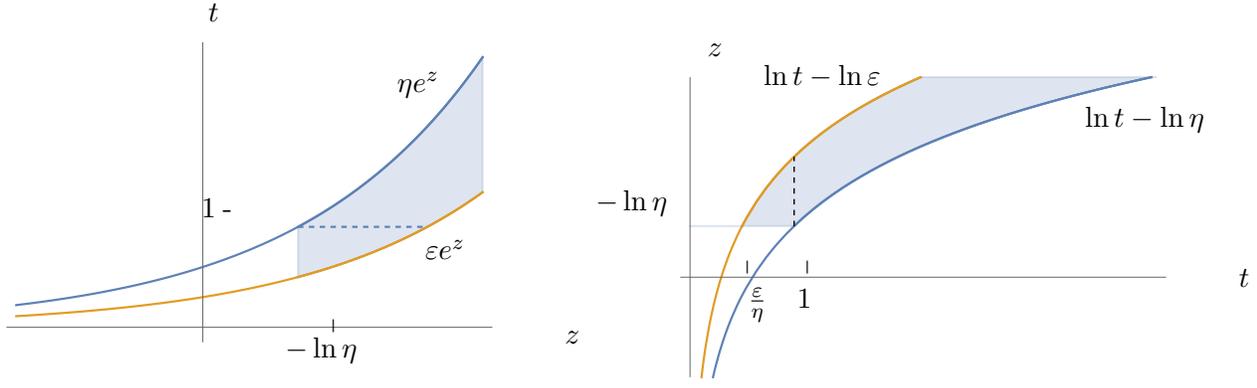

Let $B_\eta:=\{y'\in \R^{N-1}\::\: |y'|<\eta\}$. Using spherical coordinates, changes of variables ($t=\rho^{-1} y_1$,
 $dt=\rho^{-1}\, dy_1$, 
 $z=-\ln\rho$,
 $\rho=e^{-z}$, $dz = -\rho^{-1}\, d\rho$),
 and Fubini's theorem (see Figure~\ref{fig:2}),
\begin{align*}
J_{21}(\eta)
&=\int_{B_\eta}\int_\eps^\eta  \frac{(\ln 2-\ln y_1)^{-\frac{1}{2}}}{|y|^N}\, dy_1\, dy'
=\sigma_{N-1} \int_{0}^\eta\int_\eps^\eta \frac{(\ln 2-\ln y_1)^{-\frac{1}{2}}}{|y_1^2+\rho^2|^\frac{N}{2}}\rho^{N-2}\, dy_1\, d\rho\\
&=\sigma_{N-1}\int_0^\eta  \rho^{-1}  \int_{\eps\rho^{-1}}^{\eta\rho^{-1}} \frac{(\ln 2-\ln(t\rho))^{-\frac{1}{2}}}{|t^2+1|^\frac{N}{2}}\, dt\,d\rho\\
&=\sigma_{N-1}\int_{-\ln \eta}^\infty\int_{\eps e^z}^{\eta e^z} \frac{1}{|t^2+1|^\frac{N}{2}\sqrt{\ln 2-\ln t+z}}\, dt\,dz\\
&=\sigma_{N-1}\int_{\frac{\eps}{\eta}}^\infty |t^2+1|^{-\frac{N}{2}} \int_{\ln t-\ln \eta}^{\ln t-\ln\eps} \frac{1}{\sqrt{\ln 2-\ln t+z}}\, dz\,dt-R(\eta)\\
&=\sigma_{N-1}\int_{\frac{\eps}{\eta}}^\infty |t^2+1|^{-\frac{N}{2}} \left(\int_{-\ln \eta}^{-\ln\eps} \frac{1}{\sqrt{\ln 2+z}}\, dz\right)\,dt-R(\eta)\\
&=2\sigma_{N-1}\left(\sqrt{\ln 2-\ln \eps}-\sqrt{\ln 2-\ln \eta}\right)\int_{\eta^{-1}\eps}^\infty |t^2+1|^{-\frac{N}{2}}\, dt-R(\eta)\\
&=\left(\sqrt{\ln 2-\ln \eps}-\sqrt{\ln 2-\ln \eta}\right)\left(\sigma_{N-1}\frac{\sqrt{\pi } \Gamma \left(\frac{N-1}{2}\right)}{\Gamma \left(\frac{N}{2}\right)}+o(1)\right)-R(\eta)\\
&=\sigma_N\sqrt{-\ln \eps}-\sigma_N\sqrt{\ln 2-\ln \eta}-R(\eta)+o(1),
\end{align*}
because
\begin{align*}
\sigma_{N-1}\frac{\sqrt{\pi } \Gamma \left(\frac{N-1}{2}\right)}{\Gamma \left(\frac{N}{2}\right)}
&=2\frac{\pi^{\frac{N}{2}}}{\Gamma(\frac{N}{2})}=N\frac{\pi^\frac{N}{2}}{\Gamma(\frac{N}{2}+1)}=\sigma_N.
\end{align*}
The claim now follows because
\begin{align*}
0&<\sigma_{N-1}\int_{\frac{1}{2}}^{1}|t^2+1|^{-\frac{N}{2}} 
\int_{\ln t-\ln\eta}^{-\ln \eta} \frac{1}{\sqrt{\ln 2+z-\ln t}}\, dz\,dt\\
&<R(\eta)=\sigma_{N-1}\int_{\frac{\eps}{\eta}}^{1}|t^2+1|^{-\frac{N}{2}} 
\int_{\ln t-\ln\eta}^{-\ln \eta} \frac{1}{\sqrt{\ln 2+z-\ln t}}\, dz\,dt\\
&\leq \sigma_{N-1}\int_{0}^{1}
\int_{-\ln\eta}^{-\ln t-\ln \eta} \frac{1}{\sqrt{\ln 2+z}}\, dz\,dt
=2\sigma_{N-1}\int_{0}^{1}
\sqrt{\ln 2-\ln\left(\eta  t\right)}-\sqrt{\ln 2 -\ln\eta}\,dt
=O(1)
\end{align*}
as $\eps\to 0.$
\end{proof}

\begin{lemma}\label{J22:lem}
For $\eps\in(0,\frac{1}{4})$ and $\eta\in[\frac{1}{\sqrt{2}},1]$,
\begin{align*}
J_{22}(\eta):=\int_{Q_{\eps,\eta}} \frac{V(y)-V(x_\eps+y)}{|y|^N}\ dy=o(1)\qquad \text{ as }\eps\to 0.
\end{align*}
\end{lemma}
\begin{proof}
Arguing as in Lemma~\ref{J21:lem}, we have that
\begin{align*}
    &-J_{22}(\eta)
    =\sigma_{N-1} \int_0^\eta  \rho^{-1}  \int_{\eps\rho^{-1}}^{\eta \rho^{-1}} 
    \left(
    \frac{1}{
    \sqrt{\ln 2-\ln (t\rho+\eps)}}
    -
    \frac{1}{\sqrt{\ln 2-\ln(t\rho)}}\right)|t^2+1|^{-\frac{N}{2}}\, dt\,d\rho\\
    &=\sigma_{N-1}\int_{-\ln \eta}^\infty\int_{\eps e^z}^{\eta e^z} 
    \left(
    \frac{1}{\sqrt{(\ln 2-\ln(te^{-z}+\eps)}}
    -\frac{1}{\sqrt{\ln 2+z-\ln t}}\right)|t^2+1|^{-\frac{N}{2}}\, dt\,dz.
\end{align*}
Arguing as in the proof of Lemma~\ref{J2:s:aux}, this last term tends to zero as $\eps \to 0$.
\end{proof}

\begin{lemma}\label{J2}
It holds that
\begin{align}\label{J2c1}
J_2=J_2(\eps):=-\int_{B_1\backslash U_\eps} \frac{V(y+x_\eps)}{|y|^N}\ dy=-\sigma_N\sqrt{-\ln \eps}+O(1)\qquad \text{ as }\eps\to 0,
\end{align}
and
\begin{align}\label{J2c2}
J_2>-\sigma_N\sqrt{-\ln\eps}+\sigma_N\sqrt{\ln 2}+o(1)\qquad \text{ as }\eps\to 0.
\end{align}
\end{lemma}
\begin{proof}
By Lemmas~\ref{J21:lem} and~\ref{J22:lem},
\begin{align*}
-J_2<\int_{Q_{\eps,1}}\frac{V(y+x_\eps)}{|y|^N}\ dy&=J_{21}(1)-J_{22}(1)<\sigma_N\sqrt{-\ln\eps}-\omega\sqrt{\ln 2}+o(1)\qquad \text{ as }\eps\to 0,
 \end{align*}
 and~\eqref{J2c2} follows. On the other hand,
 \begin{align}\label{J2c1:1}
 J_2=-\int_{Q_{\eps,\frac{1}{\sqrt{2}}}}\frac{V(y+x_\eps)}{|y|^N}\ dy-\int_{B_1\backslash (U_\eps\cup Q_{\eps,\frac{1}{\sqrt{2}}})}\frac{V(y+x_\eps)}{|y|^N}\ dy.
 \end{align}
 Note that, by Lemmas~\ref{J21:lem} and~\ref{J22:lem},
 \begin{align}\label{J2c1:2}
 0<\int_{Q_{\eps,\frac{1}{\sqrt{2}}}}\frac{V(y+x_\eps)}{|y|^N}\ dy
 =J_{21}(1/\sqrt{2})-J_{22}(1/\sqrt{2})
 =\sigma_N\sqrt{-\ln\eps}+O(1)
 \end{align}
 as $\eps\to 0,$ whereas
 \begin{align}\label{J2c1:3}
0<\int_{B_1\backslash (U_\eps\cup Q_{\eps,\frac{1}{\sqrt{2}}})}\frac{V(y+x_\eps)}{|y|^N}\ dy
\leq 2^\frac{N}{2} \|V\|_\infty |B_1|.
 \end{align}
 But then,~\eqref{J2c1:1},~\eqref{J2c1:2}, and~\eqref{J2c1:3} imply~\eqref{J2c1}.
\end{proof}

\begin{theorem}\label{thm:bdry} Let $V$ be given by~\eqref{V}, then
\begin{align}\label{Vcl}
J:=\int_{B_1(x_\eps)} \frac{V(x_\eps)-V(y)}{|y-x_\eps|^N}\ dy>\sigma_N \sqrt{\ln 2}+o(1)
\quad \text{ and }\quad 
J=O(1)
\qquad \text{ as }\eps\to 0.
\end{align}
\end{theorem}
\begin{proof}
Let $J_1$ and $J_2$ be as in Lemmas~\ref{J1} and~\ref{J2}.  By Lemma~\ref{J1},
\begin{align*}
J&:=\int_{B_1(x_\eps)} \frac{V(x_\eps)-V(y)}{|y-x_\eps|^N}\ dy
=J_1+J_2+V(x_\eps)\int_{B_1\backslash U_\eps} |y|^{-N}\ dy+o(1)\\
&=J_1+J_2+\sigma_N \frac{-\ln(\eps)}{\sqrt{-\ln \frac{\eps}{2}}}
=J_2+\sigma_N \sqrt{-\ln \eps}+o(1)
>\sigma_N \sqrt{\ln 2}+o(1)
\end{align*}
as $\eps\to 0$. The claim~\eqref{Vcl} now follows from Lemma~\ref{J2}.
\end{proof}

Let $\widetilde V \in L^1_{loc}(\R^N)$ be given by
\begin{align}\label{Vtilde}
\widetilde V(x)=V(x)\varphi(|x|)=\frac{\varphi(x_1)\varphi(|x|)}{\sqrt{-\ln \frac{x_1}{2}}}\chi_{(0,\infty)}(x_1),\qquad x\in\R^N,
\end{align}
where $\varphi$ is given in~\eqref{phi:def}.  Recall that $B_2^+$ is the half ball of radius 2 given in~\eqref{Omega:def}.
\begin{lemma}\label{V:lem:2}
Let $\widetilde V$ be given by~\eqref{Vtilde}, then $L_\Delta \widetilde V\in L^\infty(B_2^+)$.
\end{lemma}
\begin{proof}
First observe that Theorem~\ref{thm:bdry} together with the translation invariance of $V$ $(V(x)=u(x_1))$ and the fact that $V\in C^\infty(\R^N_+)$ imply that 
\begin{align*}
\sup_{z\in B_2^+}\left|\int_{B_1}
\frac{
V(z)
-V(z+y)}{|y|^N}\, dy\right|<C
\end{align*}
for some $C>0$, as a simple argument by contradiction shows. From this estimate and Lemma~\ref{Ldelta} we deduce that $L_\Delta V \in L^\infty(B_2^+)$. On the other hand, by the Leibniz-type formula~\eqref{eq:leib_log},
\begin{equation*}
    L_\Delta \widetilde V(x) = L_\Delta V(x) \varphi(x) + V(x) L_\Delta \varphi(x) - I(V,\varphi)(x) \qquad \text{for all } x \in \R^N.
\end{equation*}
Thus, it suffices to show that $I(V,\varphi) \in L^\infty(B_2^+)$.  Observe that, for $x\in B_2^+$,
   \begin{align*}
        \left|I(V,\varphi)(x) \right| \le &\ c_{N}\int_{B_{1}(x)}\frac{\left| V(x)-V(y)\right| \left| \varphi(x)-\varphi(y)\right|}{|x-y|^N}dy\\ 
        & +c_{N} \int_{\R^N\setminus B_1(x)}\frac{\left| V(y)\varphi(y)-V(x)\varphi(y)-V(y)\varphi(x) \right|}{|x-y|^N}dy 
        +\rho_{N}V(x)\varphi(x) \\
        \leq \ &c_{N}\int_{B_{1}(x)}\frac{\left| V(x)-V(y)\right| \left| \varphi(x)-\varphi(y)\right|}{|x-y|^N}dy + c_N\int_{B_R(x)\setminus B_1(x)}\frac{\left| V(y)-V(x) \right|}{|x-y|^N} \varphi(y)dy\\ 
        &+c_{N} \varphi(x) \int_{\R^N\setminus B_1(x)}\frac{V(y)}{|x-y|^N} \, dy + \rho_{N}V(x)\varphi(x)=:A_1+A_2+A_3+A_4,
    \end{align*}
    for $R>1$ such that $B_1(x) \cup B_{1+2\zeta} \subset B_R(x)$, see~\eqref{phi:def}.  Since $V$ is bounded and $\varphi$ is smooth and bounded, it easily follows that $A_1$, $A_2$, and $A_4$ are uniformly bounded.  Furthermore, recall that $V$ has support on the strip $\{|y_1|<1+2\zeta\}$ (see~\eqref{V}). Since $x_1\in(0,2)$, we have that $\{|x_1+y_1|<1+2\zeta\}\subset \{|y_1|<4\}$ and then   
    \begin{align*}
        A_3
        &=c_N \varphi(x)\int_{\R^N\setminus B_1}\frac{V(x+y)}{|y|^N} \, dy
        \leq c_N \|V\|_\infty\|\varphi\|_\infty\int_{(\R^N\setminus B_1)\cap\{|x_1+y_1|<1+2\zeta\}}|y|^{-N} \, dy
        \\& 
        \le c_N\|V\|_\infty\|\varphi\|_\infty\left(
        \int_{\{|y'|<1\}\cap \{|y_1| < 4\}\cap \{|y|>1\}}|y|^{-N}\, dy +
        \int_{\{|y_1| < 4\}} \int_{\{|y'|>1\}} |y'|^{-N} \, dy' dy_1\right)\\
        & 
        \le c_N\|V\|_\infty\|\varphi\|_\infty\left(
        \Big|\{|y'|<1\}\cap \{|y_1| < 4\}\Big|+
        \Big|\{|y_1| < 4\}\Big|\sigma_{N-1}\int_{1}^\infty \rho^{-2} \, d \rho\right)
        < \infty,
    \end{align*}
    where $y=(y_1,y')\in \R\times \R^{N-1}$.
\end{proof}

\begin{lemma}\label{V:lem:3}
 It holds that
\begin{align*}
\int_{\R^N\backslash B_1(x_\eps)} \frac{\widetilde V(y)}{|y-x_\eps|^N}\ dy<\frac{\sigma_N}{4}\sqrt{\ln 2}\qquad \text{ for }\eps\in(0,\zeta).
\end{align*}
\end{lemma}
\begin{proof}
By~\eqref{zeta2},
\begin{align*}
&\int_{\R^N\backslash B_1(x_\eps)} \frac{\widetilde V(y)}{|y-x_\eps|^N}\ dy
=\int_{B_{1+2\zeta}\backslash B_1(x_\eps)} \frac{1}{\sqrt{-\ln \frac{y_1}{2}}}\frac{\varphi(y_1)\chi_{(0,\infty)}(y_1)}{|y-x_\eps|^N}\ dy\\
&\leq \frac{1}{\sqrt{-\ln \frac{1+2\zeta}{2}}}\int_{B_{1+2\zeta}\backslash B_1(x_\eps)} \frac{1}{|y-x_\eps|^N}\ dy
\leq \frac{1}{\sqrt{-\ln \frac{1+2\zeta}{2}}}|B_{1+2\zeta}\backslash B_1(x_\eps)|\\
&=\frac{1}{\sqrt{-\ln \frac{1+2\zeta}{2}}}|B_{1+2\zeta}\backslash B_1|
\leq \frac{1}{\sqrt{-\ln \frac{1+\frac{1}{2}}{2}}}
\frac{\sigma_N}{N}\left((1+2\zeta)^N-1\right)<\frac{\sigma_N}{4}\sqrt{\ln 2}
\end{align*}
for $\eps\in(0,\zeta).$
\end{proof}

\begin{lemma}\label{V:lem:4}
It holds that $\widetilde V\in {\mathbb H}(B_2^+)$.
\end{lemma}
\begin{proof}
We argue as in Lemma~\ref{uinH:lem}. Let $\zeta>0$ be as in~\eqref{zeta2} and define 
\begin{align*}
K:=\{x\in B_3\::\:-\zeta<x_1<\zeta\}\qquad \text{ and } \qquad U:=B_3\backslash K.
\end{align*}
 For $x\in U$, we have that
\begin{align*}
    |\widetilde V(x+y)-\widetilde V(x)|=|\nabla \widetilde V(x)\cdot y+o(y)|\leq |y|(\|\nabla \widetilde V\|_{L^\infty(U)}+o(1))\qquad \text{ as }|y|\to 0.
\end{align*}
Then
\begin{align}
A_1&:=\int_{U}\int_{B_1}\frac{(\widetilde V(y+x)-\widetilde V(x))\widetilde V(y+x)}{|y|^N} \, dy\, dx
\leq \|\widetilde V\|_\infty \int_{U}\int_{B_1} (\|\nabla \widetilde V\|_{L^\infty(U)}+o(1))|y|^{1-N}\, dy\, dx<\infty,\notag\\ \notag
A_2&:=\int_{K\cap\{x_1>0\}}\int_{B_1\cap\{y_1>x_1\}}\frac{(\widetilde V(y+x)-\widetilde V(x))\widetilde V(y+x)}{|y|^N} \, dy\, dx\\
&\leq 2\sigma_{N-1}\|\widetilde V\|^2_\infty 
\int_{K\cap\{x_1>0\}}\int_{0}^1\int_{x_1}^1|y_1^2+\rho^2|^{-\frac{N}{2}}\rho^{N-2}\,dy_1 \, d\rho\, dx\notag\\
&\leq 2\sigma_{N-1}\|\widetilde V\|^2_\infty 
\int_{K\cap\{x_1>0\}}\int_{x_1}^1y_1^{-1} \int_{0}^{\infty}\frac{\rho^{N-2}}{|1+\rho^2|^{\frac{N}{2}}}\, d\rho\,dy_1 \, dx\notag\\
&=2\sigma_{N-1}\|\widetilde V\|^2_\infty \frac{\sqrt{\pi } \Gamma \left(\frac{N-1}{2}\right)}{2 \Gamma \left(\frac{N}{2}\right)}
\int_{K\cap\{x_1>0\}}-\ln(x_1) \, dx<\infty.\label{b}
\end{align}
Moreover, 
\begin{align*}
A_3&:=\int_{K\cap\{x_1>0\}}\int_{B_1\cap\{y_1<x_1\}}\frac{(\widetilde V(y+x)-\widetilde V(x))\widetilde V(y+x)}{|y|^N} \, dy\, dx\\
&\leq \|\widetilde V\|_\infty\int_{K\cap\{x_1>0\}}\int_{B_1\cap\{-x_1<y_1<x_1\}}\frac{|\widetilde V(x)-\widetilde V(y+x)|}{|y|^N} \, dy\, dx<\infty,
\end{align*}
where the finiteness of $A_3$ follows from~\eqref{lem211} (note that~\eqref{lem211} is stated for $V$ instead of $\widetilde V$, but since $\widetilde V$ is given by~\eqref{Vtilde}, the bound easily extends to $\widetilde V$).

Furthermore, since $\widetilde V(x+y)=0$ if $x_1+y_1<0$, arguing as in~\eqref{b},
\begin{align*}
A_4&
:=\int_{K\cap\{x_1<0\}}\int_{B_1}\frac{(\widetilde V(y+x)-\widetilde V(x))\widetilde V(y+x)}{|y|^N} \, dy\, dx\\
&=\int_{K\cap\{x_1<0\}}\int_{B_1\cap\{-x_1<y_1<1\}}\frac{(\widetilde V(y+x)-\widetilde V(x))\widetilde V(y+x)}{|y|^N} \, dy\, dx\\
&\leq 2\sigma_{N-1}\|\widetilde V\|^2_\infty\int_{K\cap\{x_1>0\}}dx\int_{x_1}^1\int_0^1 \frac{\rho^{N-2}}{|y_1^2+\rho^2|^{\frac{N}{2}}} \,d\rho\, dy_1\, dx
\\
& \leq 2\sigma_{N-1}\|\widetilde V\|^2_\infty\int_{K\cap\{x_1>0\}} dx \int_{x_1}^{1}y_1^{-1}dy_1\int_0^\infty \frac{r^{N-2}}{|1+r^2|^{N/2}}dr\\
&= C \|\widetilde V\|^2_\infty \int_{K\cap\{x_1>0\}}-\ln(x_1) dx < \infty,
\end{align*}
where $C>0$ is a constant only depending on $N$.

Then,
\begin{align}\label{AsN}
A:=\int_{\R^N}\int_{B_1}\frac{(\widetilde V(y+x)-\widetilde V(x))\widetilde V(y+x)}{|y|^N}
    \, dy\, dx=A_1+A_2+A_3+A_4<\infty.
\end{align}

Recall that
\begin{align*}
&\frac{2}{c_N}\|\widetilde V\|^2
=\int_{\R^N}\int_{\R^N}
    \frac{(\widetilde V(x)-\widetilde V(y))\widetilde V(x)}{|x-y|^N}\chi_{B_1(x)}(y)
    +\frac{(\widetilde V(y)-\widetilde V(x))\widetilde V(y)}{|x-y|^N}\chi_{B_1(x)}(y)
    \, dy\, dx.
\end{align*}
Therefore, using Fubini's theorem, a change of variables,
and~\eqref{AsN},
    \begin{align*}
&\int_{\R^N}\int_{\R^N}
    \frac{(\widetilde V(x)-\widetilde V(y))\widetilde V(x)}{|x-y|^N}\chi_{B_1(x)}(y)
    \, dy\, dx
    =\int_{\R^N}\int_{\R^N}
    \frac{(\widetilde V(x)-\widetilde V(y))\widetilde V(x)}{|x-y|^N}\chi_{B_1(x)}(y)\, dx\, dy\\
&\qquad =\int_{\R^N}\int_{B_1}
    \frac{(\widetilde V(z+y)-\widetilde V(y))\widetilde V(z+y)}{|z|^N}\, dz\, dy=A
    <\infty.
\end{align*}
Similarly, 
\begin{align*}
    \int_{\R^N}\int_{\R^N}
    \frac{(\widetilde V(y)-\widetilde V(x))\widetilde V(y)}{|x-y|^N}\chi_{B_1(x)}(y)
    \, dy\, dx
    =\int_{\R^N}\int_{B_1}\frac{(\widetilde V(z+x)-\widetilde V(x))\widetilde V(z+x)}{|z|^N}
    \, dz\, dx=A<\infty.
\end{align*}
Then $\|\widetilde V\|^2<\infty$.  Since we also know that $\widetilde V\in L^2(\R^N)$, we obtain that $\widetilde V\in \mathbb H(B_2^+)$ as claimed.
\end{proof}

The following result implies Theorem~\ref{barrier:thm} in the case $N\geq 2$.

\begin{theorem}\label{thm:log:N} The function $\widetilde V$ given by~\eqref{V} belongs to ${\mathbb H}(B_2^+)\cap C^\infty(\R^N_+)\cap\cL^{\frac{1}{2}}(\R^N)$ and there is $\delta\in(0,2)$ such that
\begin{align*}
L_\Delta \widetilde V\in L^\infty(B_2^+)\qquad \text{ and }\qquad L_\Delta \widetilde V(x)>\frac{\sigma_N}{2}\sqrt{\ln 2}\quad \text { for }x\in B_\delta\cap\R^N_+.
\end{align*}
\end{theorem}
\begin{proof}
By Lemma~\ref{V:lem:2} we know that $L_\Delta \widetilde V\in L^\infty(B_2^+)$.  By Lemma~\ref{V:lem:4} and by construction we have that $\widetilde V \in {\mathbb H}(B_2^+)\cap C^\infty(\R^N_+)\cap \cL^{\frac{1}{2}}(\R^N)$. Finally, by Theorem~\ref{thm:bdry},
\begin{align*}
J:=\int_{B_1(x_\eps)} \frac{V(x_\eps)-V(y)}{|y-x_\eps|^N}\ dy>\sigma_N \sqrt{\ln 2}+o(1)\qquad \text{ as }\eps\to 0.
\end{align*}
Since, by~\eqref{phi:def}, it holds that $\varphi(|x|)=1$ for $x\in B_{1+\zeta}$ , we have that
$\widetilde V=V$ in $B_{1+\zeta}$ (with $V$ given in~\eqref{V}) and therefore
\begin{align*}
\int_{B_1(x)} \frac{\widetilde V(x)- \widetilde V(y)}{|y-x|^N}\ dy
=\int_{B_1(x)} \frac{V(x)-V(y)}{|y-x|^N}\ dy\qquad \text{ for all }x\in B_\zeta.
\end{align*}
In particular,
\begin{align}\label{VwV}
\int_{B_1(y_\eps)} \frac{\widetilde V(y_\eps)- \widetilde V(y)}{|y-y_\eps|^N}\ dy
=J>\sigma_N \sqrt{\ln 2}+o(1)\qquad \text{ as }\eps\to 0
\end{align}
for all $y_\eps=(\eps,y')$ where $y'\in\R^{N-1}$ is such that $y_\eps\in B_\zeta$. Then, by Lemma~\ref{V:lem:3},~\eqref{VwV} and Lemma~\ref{V:lem:1} (with $\sigma = \sigma_N\sqrt{\ln 2}/2$ and $\eta=3/4$), there is $\delta\in(0,\zeta)$ such that 
\begin{align*}
L_\Delta \widetilde V(y_\eps)>\frac{\sigma_N}{2}\sqrt{\ln 2}
\end{align*}
for all $y_\eps=(\eps,y')\in B_\delta\cap B_2^+$.
\end{proof}

\begin{proof}[Proof of Theorem~\ref{barrier:thm}] The claim follows from Theorems~\ref{thm:L} (for the case $N=1$) and Theorem~\ref{thm:log:N} (for the case $N\geq 2$).
\end{proof}

\section{A new barrier via the Kelvin transform} \label{sec:kelvin}

Let $R>0$ and $x_0\in \R^N$. The inversion of a point $x\in \R^N$ with respect to the sphere $S_{R}(x_0)$ is given by
\begin{equation} \label{eq:inversion_sphere}
    x^* := x_0 + R^2\frac{x-x_0}{|x-x_0|^2}, \quad x \neq x_0,
\end{equation}
and the Kelvin transform of $u$ is
\begin{equation} \label{eq:kelvin_transform}
    u^*(x^*) := |x^*-x_0|^{-N} u(x), \quad x^* \neq x_0.
\end{equation}
For $\Omega \subset \R^N$ we define $\Omega^* := \{ x^*: x\in \Omega \}$.  Recall that
\begin{align}\label{Kelvin:prop}
    |x^*-z^*|=\frac{|x - z|}{|x||z|}\qquad \text{ for }x,z\in\R^N\backslash\{0\},
\end{align}
see, for instance, \cite[Proposition A.3]{B16}.  

To prove Proposition~\ref{K:lem}, we show an auxiliary lemma first. Recall the definition of $h_\Omega:\Omega\to\R$ (see~\eqref{eq:h_omega}) given by
\begin{equation*}
h_\Omega(x):=c_N\left(\int_{B_1(x)\backslash \Omega}|x-y|^{-N}\, dy-\int_{ \Omega\backslash B_1(x)}|x-y|^{-N}\, dy\right).
\end{equation*}

\begin{lemma}\label{h:aux:lem}
Let $R,$ $x_0$, and $x^*$ as in~\eqref{eq:inversion_sphere}. Let $\Omega\subset\R^N$ be an open bounded subset such that $B_R(x_0)\subset \Omega^c.$ The function $x\mapsto h_{\Omega^*}(x^*)-h_\Omega(x)$ is bounded in $\Omega$.
\end{lemma}
\begin{proof}
Without loss of generality we assume that $x_0=0$ and $R=1$; in particular this implies that $|x|>1$ for every $x\in \Omega$.  Given $x \in \Omega$ consider $B_{1/2}(x)$ that, under the inversion with respect to the unit sphere, is transformed into the ball $B_\rho(P)$, where
\begin{equation*}
    P = \frac{4x}{4|x|^2-1}, \quad \rho = \frac{2}{4|x|^2-1},
\end{equation*}
see \cite[Section 2]{BZ06}. This ball contains another one centered at $x^*$ such that
\begin{equation}
    \label{eq:ball_contention}
    B_\sigma(x^*) \subset B_\rho(P),
\end{equation}
where 
\begin{align}\label{sigma:def}
\sigma = \frac{2|x|-1}{|x|(4|x|^2-1)}.
\end{align}
Indeed, if $y \in B_\sigma(x^*)$, then
\begin{equation*}
    |y-P| \le |y-x^*|+|x^*-P| < \sigma + \frac{1}{|x|(4|x|^2-1)} = \rho.
\end{equation*}
Observe that $0<\sigma< \rho < 1$, as $|x| \ge 1$. 

Let $A_1 := [B_1(x) \setminus (B_{1/2}(x) \cup \Omega)] \cup [\Omega \setminus B_1(x)]$ and
\begin{equation*}
    R_1(x) := c_N \int_{A_1} \frac{1}{|y-x|^N} \, dy.
\end{equation*}
Note that $R_1\in L^\infty(\Omega)$. Then, by a change of variables ($y=z^*$, see \cite[Proposition A.3]{B16}) and~\eqref{Kelvin:prop},
\begin{equation}
    \label{eq:h_omega_est}
    \begin{aligned}
    h_\Omega(x) &= c_N\int_{B_{1/2}(x) \setminus \Omega} \frac{1}{|y-x|^N} \, dy + R_1(x) \\
    &= c_N \int_{B_\rho(P) \setminus \Omega^*} \frac{|z|^{-2N}}{|z^*-x|^N} \, dz + R_1(x)
    = c_N \int_{B_\rho(P) \setminus \Omega^*} \frac{|z|^{-N}|x^*|^N}{|z-x^*|^N} \, dz + R_1(x).
    \end{aligned}
\end{equation}
On the other hand, let $A_2 = [B_1(x^*) \setminus (B_\sigma(x^*) \cup \Omega)] \cup [\Omega \setminus B_1(x^*)]$ (with $\sigma$ as in~\eqref{sigma:def}) and 
\begin{equation*}
    R_2(x^*) = c_N \int_{A_2} \frac{1}{|z-x^*|^N} \, dz.
\end{equation*}
Note that $R_2\in L^\infty(\Omega^*)$ and 
\begin{equation}
    \label{eq:h_omega*_est}
    h_{\Omega^*}(x^*) = c_N \int_{B_\sigma(x^*) \setminus \Omega^*} \frac{1}{|z-x^*|^N} \, dz + R_2(x^*).
\end{equation}

Let $A_3 := [B_\rho(P) \setminus (B_\sigma(x^*) \cup \Omega^*)]$ and 
\begin{equation*}
    R_3(x) := |R_1(x) - R_2(x)|+ c_N \int_{A_3} \left| \frac{|z|^{-N}|x^*|^N}{|z-x^*|^N} \right| \, dz.
\end{equation*}
Note that $R_3\in L^\infty(\Omega)$ and, from~\eqref{eq:ball_contention},~\eqref{eq:h_omega_est}, and~\eqref{eq:h_omega*_est} we deduce that
\begin{align}
        |h_\Omega(x)-h_{\Omega^*}(x^*)| 
        &\le \left|\int_{B_\sigma(x^*) \setminus \Omega^*} \frac{|z|^{-N}|x^*|^N-1}{|z-x^*|^N} \, dz \right|+ R_3(x) \notag\\
        &\le \int_{B_\sigma(x^*)} \frac{1}{|z|^N} \frac{\left| |z|^N-|x^*|^N \right|}{|z-x^*|^N} \, dz + R_3(x) \notag\\
        &\le \left( \frac{4C^2-1}{2} \right)^N \int_{B_\sigma(x^*)} \frac{\left| |z|^N-|x^*|^N \right|}{|z-x^*|^N} \, dz + R_3(x).\label{eq_h_omega-h_omega*}
\end{align}
For the last inequality in~\eqref{eq_h_omega-h_omega*}, we used that $B_\sigma(x^*) \subset \R^N \setminus B_{2/(4C^2-1)}$ with $C := \sup_{x \in \Omega} |x|$.

To conclude, we use polar coordinates to deduce that, up to a constant,
\begin{equation} \label{eq:z-xi}
    \int_{B_\sigma(x^*)} \frac{\left| |z|^N-|x^*|^N \right|}{|z-x^*|^N} \, dz = \int_0^\sigma \int_{S_1} \frac{\left| |x^*+r\theta|^N-|x^*|^N \right|}{r} \, dS(\theta) dr,
\end{equation}
where $S_1 = \partial B_1$. By a Taylor expansion,
\begin{equation}\label{T:exp}
    \left|\frac{|x^*+r\theta|^N-|x^*|^N}{r}\right|\leq N |\theta||x^*|^{N-1} + R(r),
\end{equation}
where $R$ is continuous and $R(r) \to 0$ as $r \to 0^+$. Then~\eqref{eq_h_omega-h_omega*},~\eqref{eq:z-xi}, and~\eqref{T:exp} yield that $h_\Omega - h_{\Omega^*} \in L^\infty(\Omega)$.
\end{proof}

\begin{proof}[Proof of Proposition~\ref{K:lem}]
Without loss of generality we assume that $x_0=0$ and $R=1$.

To prove $i)$ let $u$ be a Dini continuous function at $x$ and $x\neq x_0$.  We argue as in \cite[Proposition A.1]{RS14}; however the logarithmic case is more involved.  In particular, one cannot assume without loss of generality that $u^*(x^*)=0$. By \cite[Proposition 2.2]{CW19},
\begin{equation} \label{eq:L_delta2}
L_\Delta u^*(x^*)=c_N\int_{\Omega^*}\frac{u^*(x^*)- u^*(y)}{|x^*-y|^N}\, dy + (h_{\Omega^*}(x^*)+\rho_N) |x^*|^{-N}u(x),
\end{equation}
where $h_{\Omega^*}$ is given by~\eqref{eq:h_omega}. Using the change of variables $y=z^*$ (see \cite[Proposition A.3]{B16}) and~\eqref{eq:L_delta2},
  \begin{align*}
c_N \int_{\Omega^*} &\frac{u^*(x^*) - u^*(y)}{|x^*-y|^N}\, dy= c_N \int_{\Omega} \frac{u^*(x^*) - u^*(z^*)}{|x^*-z^*|^N} |z|^{-2N}\, dz \notag\\
&= c_N\int_{\Omega} \frac{|x^*|^{-N} u(x) - |z^*|^{-N} u(z)}{|x-z|^N}|x^*|^{-N}|z^*|^{-N} |z^*|^{2N}\, dz \notag\\
&=c_N |x^*|^{-2N} \int_{\Omega} \frac{|z^*|^Nu(x) - |x^*|^Nu(z)}{|x-z|^N}\, dz \notag\\
&= c_N |x^*|^{-2N}u(x) \int_{\Omega}\frac{|z|^{-N}-|x|^{-N}}{|z-x|^N}\, dz
+ c_N |x^*|^{-N} \int_{\Omega}\frac{u(x)-u(z)}{|x-z|^N}\, dz  \notag\\
&= c_N |x^*|^{-2N}u(x) \int_{\Omega}\frac{|z|^{-N}-|x|^{-N}}{|z-x|^N}\, dz + |x^*|^{-N} L_\Delta u(x) - (h_\Omega(x)+\rho_N)|x^*|^{-N}u(x).
 \end{align*}
 This equation and~\eqref{eq:L_delta2} imply~\eqref{eq:L_delta1}.

To show $ii)$, assume that $u,L_\Delta u \in L^\infty(\Omega)$ and that 
\begin{align}\label{b:hyp}
B_R(x_0)=B_1(0) \subset \Omega^c.    
\end{align}
 By~\eqref{b:hyp}, it is clear that $u^*\in L^\infty(\Omega^*)$. Moreover, by Lemma~\ref{h:aux:lem},~\eqref{b:hyp}, and~\eqref{eq:L_delta1}, it also follows that $L_\Delta u^* \in L^\infty (\Omega^*)$, as claimed.

To prove $iii)$, let $u \in \mathbb{H}(\Omega)$ and recall that we have assumed that $x_0=0$, $R=1$, and that~\eqref{b:hyp} holds.  Note that, in this setting, $\Omega^*\subset B_1$, $B_2(y)\subset B_3$ for all $y\in \Omega^*$, and $B_2(y)\subset \R^N\backslash \Omega^*$ for all $y\in \R^N\backslash B_3$. To ease notation, let $S=S(x,y):=\frac{\left( u^*(x)-u^*(y)\right)^2}{|x-y|^N}$, then, using that $u^*=0$ in $\R^N\backslash \Omega^*$,
\begin{align*}
\int_{\R^N}\int_{B_2(y)}S \, dx dy
&=\int_{\Omega^*}\int_{B_2(y)}S \, dx dy
+\int_{\R^N\backslash \Omega^*}\int_{B_2(y)}S \, dx dy\\
&=\int_{\Omega^*}\int_{\Omega^*}S \, dx dy
+\int_{\Omega^*}\int_{B_2(y)\backslash \Omega^*}S \, dx dy
+\int_{B_3\backslash \Omega^*}\int_{B_2(y)}S \, dx dy.
\end{align*}

Note that, by~\eqref{b:hyp}, there is $\delta>0$ such that $B_\delta\subset (\Omega^*)^c$ and therefore there is some $M>0$ such that 
\begin{align*}
(B_3\backslash \Omega^*)^*\subset A\backslash \Omega,\qquad A:=B_M\backslash B_\frac{1}{3}.
\end{align*}
Then, by a change of variables ($x = \xi^*$, $y=z^*$) and~\eqref{Kelvin:prop}, 
\begin{align}
    &\int_{\R^N}\int_{B_2(y)} \frac{\left( u^*(x)-u^*(y)\right)^2}{|x-y|^N} \, dx dy 
    \le \int_{\Omega^*} \int_{\Omega^*} \frac{\left( u^*(x)-u^*(y)\right)^2}{|x-y|^N} \, dx dy + 2 \int_{\Omega^*} \int_{B_3 \setminus \Omega^*} \frac{u^*(y)^2}{|x-y|^N} \, dx dy \notag\\
    &\le \int_\Omega \int_\Omega \frac{\left( |\xi|^N u(\xi)-|z|^N u(z) \right)^2}{|\xi-z|^N} |\xi|^{-N} |z|^{-N} \, d\xi dz 
    + 2 \int_\Omega \int_{A\backslash \Omega}\frac{|z|^{2N} u(z)^2}{|\xi-z|^N} |\xi|^{-N} |z|^{-N} \, d\xi dz.\label{P1}
\end{align}
Using that $\Omega$ is bounded,~\eqref{b:hyp}, and that $(|\xi|^N u(\xi)-|z|^N u(z))^2\leq 2(|\xi|^{2N} (u(\xi)-u(z))^2+2(|\xi|^N-|z|^N)^2 u(z)^2$, we have that
\begin{align}
&\int_\Omega \int_\Omega \frac{\left( |\xi|^N u(\xi)-|z|^N u(z) \right)^2}{|\xi-z|^N} |\xi|^{-N} |z|^{-N} \, d\xi dz\notag\\
&\leq 2\int_\Omega \int_\Omega 
\frac{(u(\xi)-u(z))^2}{|\xi-z|^N} |\xi|^{N} |z|^{-N}
+\frac{(|\xi|^N-|z|^N)^2}{|\xi-z|^N}  u(z)^2|\xi|^{-N} |z|^{-N} \, d\xi dz\notag\\
&\leq C\int_\Omega \int_\Omega 
\frac{(u(\xi)-u(z))^2}{|\xi-z|^N} \, d\xi dz
+C_1\int_\Omega  u(z)^2 \, dz\label{P2}
\end{align}
for some $C_1=C_1(\Omega)>0$, where we used that
\begin{align*}
\sup_{a,b\in \Omega}\left(\frac{|a|^N-|b|^N}{|a-b|}\right)^2|a|^{-N}|b|^{-N}<\infty\quad \text{ and }\quad \sup_{a\in\Omega}\int_\Omega |\xi-a|^{2-N}\, d\xi<\infty.
\end{align*}
Moreover, using~\eqref{b:hyp} and that $\Omega$ is bounded, we find some $C_2=C_2(\Omega)>0$ such that
\begin{align}
\int_\Omega \int_{A\backslash \Omega}\frac{u(z)^2}{|\xi-z|^N}  \frac{|z|^{N}}{|\xi|^{N}} \, d\xi dz
&< C_2 \int_\Omega u(z)^2 \int_{A\backslash \Omega} |\xi-z|^{-N}\, d\xi dz\notag\\
&= \frac{C_2}{c_N} \int_\Omega u(z)^2 \kappa_\Omega(z)dz
+C_2 |A|\int_\Omega u(z)^2  dz,\label{P3}
\end{align}
where $\kappa_\Omega$ is the so-called \emph{killing measure} (see \cite[eq. (3.5)]{CW19}) given by 
$\kappa_{\Omega}(z):=c_N \int_{B_1(z) \backslash \Omega} |z-y|^{-N} dy$ for $z\in \Omega.$  By the logarithmic boundary Hardy inequality (see \cite[Remark 4.3 and Corollary A.2]{CW19}), we have that $\int_\Omega u(z)^2 \kappa_\Omega(z)dz<\infty$.  Claim $iii)$ now follows from~\eqref{P1},~\eqref{P2}, and~\eqref{P3}.
\end{proof}

\begin{remark}
The proof of the previous lemma can be slightly adapted to study the properties of functions in the fractional Sobolev spaces
\begin{equation*}
{\mathcal H}_0^s(\Omega):=\left\lbrace u\in L^{2}(\mathbb{R}^{N})~:~\int\int_{\R^N \times \R^N}\frac{|u(x)-u(y)|^2}{|x-y|^{N+2s}}\, dx\, dy<\infty~\mbox{and}~u=0~\mbox{in}~\mathbb{R}^{N}\setminus\Omega\right\rbrace,
\end{equation*}
$s \in (0,1)$, under their corresponding Kelvin transform
\begin{equation*}
    u^*(x^*) = |x^*-x_0|^{2s-N} u(x),
\end{equation*}
see \cite{RS14}. The following result is probably well-known, but we could not locate it in the literature.  We state it here for future reference. 

\begin{lemma}
    Let $\Omega$ be a bounded open set and $B_R(x_0) \subset \Omega^c$. If $u \in {\mathcal H}_0^s(\Omega)$, then $u^* \in {\mathcal H}_0^s(\Omega^*)$. 
\end{lemma}
\end{remark}

The Kelvin transform can now be used to produce a new barrier that will be useful to show the optimal regularity of the torsion function, see Figure~\ref{fig:kelvin}.

\begin{proposition}\label{v:barrier:prop}
Let $N\geq 2$ and $D:=B_{\frac{1}{2}}(\frac{1}{2}e_1)\cap\{x_1>\frac{1}{2}\}=\{x\in\R^N\::\: |x-\frac{1}{2}|<\frac{1}{2},\, x_1>\frac{1}{2}\}$. There are $f\in L^\infty(D)$ and $v\in \cL^{\frac{1}{2}}(\R^N)\cap C^\infty(D)\cap {\mathbb H}(D)$ such that $L_\Delta v= f$ pointwisely in $D$ and there are $c>0$ and $\delta>0$ such that
\begin{align}\label{v:est}
c\,\ell^{\frac{1}{2}}(\operatorname{dist}(x,\partial D)) < v(x)<c^{-1}\ell^{\frac{1}{2}}(\operatorname{dist}(x,\partial D))\qquad \text{ for all }x\in D\cap B_\delta(e_1).
\end{align}
\end{proposition}
\begin{proof}Let $w$ be given by Theorem~\ref{barrier:thm} and $\Omega:=B_2\cap \R^N_+$. Let $\widetilde w(x):=w(2(x-e_1))$ and $\widetilde \Omega:=\{x\in B_1(e_1)\::\: x_1>1\}$. Then, by Theorem~\ref{barrier:thm} and Lemmas~\ref {prop1} and~\ref{scaling:lem:class}, we have that $\widetilde w\in \cL^{\frac{1}{2}}(\R^N)\cap C^\infty(\widetilde \Omega)\cap {\mathbb H}(\widetilde \Omega)$ is a solution of $L_\Delta \widetilde w= \widetilde f$ in $\widetilde \Omega$ for some $\widetilde f\in L^\infty(\widetilde \Omega)$  and there are $c_1>0$ and $\delta>0$ such that
\begin{align*}
c_1\,\ell^{\frac{1}{2}}(\operatorname{dist}(x,\partial \widetilde\Omega)) < \widetilde w(x) <c_1^{-1}\ell^{\frac{1}{2}}(\operatorname{dist}(x,\partial \widetilde\Omega))\qquad \text{ for all }x\in \widetilde \Omega\cap B_\delta(e_1).
\end{align*}
Let $\kappa:\R^N\to\R^N$ be the Kelvin transform $\kappa(x):=\frac{x}{|x|^2}$, that is, the point inversion with respect to the sphere $S_1$. For $x\neq 0$, let $v(x):=\widetilde w(\kappa(x))$ and let $D$ be as in the statement.  It is not hard to see that $\kappa(\widetilde\Omega)=D$ and that
\begin{align*}
\kappa(\partial \widetilde\Omega\cap \{x\in \R^N\::\: x_1=1\})=\partial D\cap \partial B_\frac{1}{2}(\tfrac{1}{2}e_1),
\end{align*}
namely, that the flat part of the boundary of $\widetilde\Omega$ goes to the curved part of the boundary of $D$ through $\kappa$ (see Figure~\ref{fig:kelvin}).  Using Proposition~\ref{K:lem}, it follows that $v\in C^\infty(D)\cap {\mathbb H}(D)$ is a solution of $L_\Delta v= f$ in $D$ for some $f\in L^\infty(D).$  A simple calculation using Lemma~\ref{prop1} and~\eqref{Kelvin:prop} yields that $v\in\cL^{\frac{1}{2}}(\R^N)$ and that~\eqref{v:est} holds.
\end{proof}

We close this subsection with the following result on weak solutions under the Kelvin transform. 

\begin{lemma} \label{weak:lem}
Let $\Omega \subset \R^N$ be as in Proposition~\ref{K:lem} $ii)$.  Let $u$ be a weak solution of
\begin{align*}
    L_\Delta u = f\quad \text{ in }\Omega,\qquad u=0\quad \text{ in }\R^N \setminus \Omega,
\end{align*}
Then $u^*$ is a weak solution of
\begin{align*}
    L_\Delta u^* = \bar f\quad \text{ in }\Omega^*,\qquad u^*=0\quad \text{ in }\R^N \setminus \Omega^*,
\end{align*}
where
\begin{align*}
\bar f(y):=f(y^*)|y|^{-N}
+ c_N u(y^*) |y|^{-2N}\int_{\Omega}\frac{|z|^{-N}-|y^*|^{-N}}{|z-y^*|^N}\, dz + u(y^*)(h_{\Omega^*}(y)-h_\Omega(y^*))|y|^{-N},\quad y\in \Omega^*.
\end{align*}
\end{lemma}
\begin{proof}
Let $x_0\in\mathbb R^N$ be as in the statement of Proposition~\ref{K:lem} and without loss of generality assume that $x_0=0$. Let $\psi\in C_c^\infty(\Omega^*)$ and let $\varphi\in C^\infty_c(\Omega)$ be such that $\varphi^*(y)=|y|^{-N}\varphi(y^*)=\psi(y)$ for $y\in \Omega^*$. By definition of weak solution and~\eqref{eq:725}, we have that
    \begin{equation}
\mathcal E(u^*,\psi)=\int_{\Omega^*}u^*(y)L_{\Delta}\psi(y) \, dy
=\int_{\Omega^*}u^*(y)L_{\Delta}\varphi^*(y) \, dy. 
    \end{equation}
Then, by Proposition~\ref{K:lem},
\begin{align*}
L_\Delta \varphi^*(y)
= |y|^{-N} L_\Delta \varphi(y^*) + c_N |y|^{-2N} \varphi(y^*) \int_{\Omega}\frac{|z|^{-N}-|y^*|^{-N}}{|z-y^*|^N}\, dz  + (h_{\Omega^*}(y)-h_\Omega(y^*))|y|^{-N}\varphi(y^*)
\end{align*}
for $y\in \Omega^*$.  Thus, using that $u^*(y) = |y|^{-N} u(y^*)$ and $\psi(y)=|y|^{-N}\varphi(y^*)$,
\begin{align*}
\mathcal E(u^*,\psi)
&=\int_{\Omega^*}
u^*(y)|y|^{-N} L_\Delta \varphi(y^*) 
+ c_N u^*(y) |y|^{-2N} \varphi(y^*) \int_{\Omega}\frac{|z|^{-N}-|y^*|^{-N}}{|z-y^*|^N}\, dz \\
&\qquad + u^*(y)(h_{\Omega^*}(y)-h_\Omega(y^*))|y|^{-N}\varphi(y^*)
\, dy\\
&=\int_{\Omega^*}
u(y^*)|y|^{-2N} L_\Delta \varphi(y^*) 
+ c_N u(y^*) |y|^{-2N} \psi(y) \int_{\Omega}\frac{|z|^{-N}-|y^*|^{-N}}{|z-y^*|^N}\, dz \\
&\qquad + u(y^*)(h_{\Omega^*}(y)-h_\Omega(y^*))|y|^{-N}\psi(y^*)
\, dy.
\end{align*}
Using that $u$ is a weak solution and the change of variables $y=x^*$, we obtain that
\begin{align*}
\int_{\Omega^*}
u(y^*)|y|^{-2N} L_\Delta \varphi(y^*)\, dy
=\int_{\Omega}u(x)L_\Delta \varphi(x)\, dx=\cE(u,\varphi)=\int_\Omega f(x) \varphi(x)\, dx
=\int_{\Omega^*}
\frac{f(y^*)}{|y|^{N}}\psi(y) \, dy
.
\end{align*}
The claim now follows. 
\end{proof}

\section{Optimal boundary bounds for solutions}\label{ob:sec}

Recall that $\Psi$ denotes the digamma function, $\gamma$ is the Euler-Mascheroni constant, and, for a bounded open set $\Omega$, the space ${\mathcal V}(\Omega)$ is defined in~\eqref{Vspace:def}.  We use the following result shown in \cite[Corollary 1.9]{CW19}.  
\begin{theorem}[Weak maximum principle for the logarithmic Laplacian]\label{mp:thm}
Let $\Omega \subset \R^N$ be an open bounded domain such that $|\Omega|<2^Ne^{\frac{N}{2} \left( \Psi \left( \frac{N}{2} \right)-\gamma \right)}|B_1|$ and let $u\in {\mathcal V}(\Omega)$ be such that
\begin{align*}
\cE_L(u,\varphi)\geq 0,\qquad u\geq 0\quad \text{ in }\R^N\backslash \Omega,
\end{align*}
for all nonnegative $\varphi\in C^\infty_c(\Omega)$.  Then $u\geq 0$ a.e. in $\R^N$.
\end{theorem}

Recall that $B_\eps$ is the ball in $\R^N$ centered at 0 of radius $\eps$ and $B_\eps^+:=\{x\in B_\eps\::\: x_1>0\}$.  We are ready to show Theorem~\ref{convex:thm}.

\begin{proof}[Proof of Theorem~\ref{convex:thm}]
Let $w$, $c$, and $\delta$ as in Theorem~\ref{barrier:thm} and let $u$ be as in the statement of the Theorem. First observe that, by Lemma~\ref{scaling:lem} and Lemma~\ref{prop1}, it is enough to prove~\eqref{eq:boundary_bound} for $\Omega \subset B_\delta$; in particular $L_\Delta w > c$ in $\Omega\cap \R^N_+$. Moreover, by making $\delta$ smaller, if necessary, we may assume that 
\begin{align*}
|\Omega|<2^Ne^{\frac{N}{2} \left(\Psi \left(\frac{N}{2} \right) -\gamma \right)}|B_1|.
\end{align*}
Let $x_0 \in \partial \Omega$. We can assume $x_0=0$ and the tangent space to $\partial \Omega$ in $x_0$ to be $\partial \R_+^N$.  Let $\eta$ be as in~\eqref{eta:def}. By making $\eta$ smaller, if necessary, choose $y_0=(\eta,0,\dots,0) \in \R_+^N$, $\eta<\delta/2$ such that $\overline{B_{\eta}(y_0)} \cap \overline{\Omega} = \{ x_0 \}$.

We now consider the inversion with respect to the sphere $S_{\eta}(y_0)$ and the Kelvin transform of $u$,
\begin{equation*}
    u^*(x^*) = |x^*-y_0|^{-N} u(x), \quad x^* \in \Omega^*,
\end{equation*}
see~\eqref{eq:inversion_sphere},~\eqref{eq:kelvin_transform} (with $R=\eta$ and $x_0=y_0$). Observe that $\Omega^* \subset B_\delta^+$ and $y_0\not\in \Omega^*$.

Let $v:=Mw-u^*$ with $M=c^{-1}\|\overline{f}\|_{L^\infty(\Omega)}$ and $\overline{f}$ as in Lemma~\ref{weak:lem}.  Since $u^*\in C(\R^N)\cap \cH(\Omega^*)$, by Proposition~\ref{K:lem} and \cite[Theorem 1.11]{CW19}, and $w\geq 0$ in $\R^N$, we have that $v\geq 0$ in $\R^N\backslash\Omega^*$. Moreover, for every $\varphi\in C^\infty_c(\Omega^*)$ nonnegative,
\begin{align*}
\cE_L(v,\varphi)\geq  \int_{\Omega^*} (Mc-\|\overline{f}\|_{L^\infty(\Omega)})\varphi\, dx\geq 0,
\end{align*}
where we have used that
$\cE_L(w,\varphi)
=\int_{\Omega^*} L_\Delta w(x)\varphi(x)\, dx\geq c\int_{\Omega^*}\varphi(x)\, dx.
$ Therefore, by Theorem~\ref{mp:thm}, $v\geq0$ in $\R^N$; in particular,
\begin{equation*}
 u^*(x^*) \leq M w(x^*) \leq C \ell^\frac{1}{2}(x_1^*)
\end{equation*}
for some $C>0$ and for all $x^*$ in the segment joining $x_0$ and $y_0$, by Theorem~\ref{barrier:thm}. From this inequality we deduce that for every $x \in \Omega$ that belongs to the segment joining $-y_0$ and $x_0$,
\begin{equation} \label{eq:bound_above}
    u(x) = r_0^{-N}(r_0-x_1^*)^N u^*(x_1^*) \leq C \ell^\frac{1}{2}(|x_1|)= C\sqrt{\ell(\dist(x,\partial \Omega))},
\end{equation}
as $x_1^* \le |x_1|$.

We conclude that~\eqref{eq:bound_above} is true for all $x$ in the $\eta$-neighborhood of $\partial \Omega$.  A similar argument can now be done using $-u$ instead of $u$, which yields that~\eqref{eq:bound_above} holds for $-u$ too. This implies~\eqref{eq:boundary_bound}, as claimed.
\end{proof}

\section{Applications}\label{app:sec}

\subsection{Optimal regularity for the torsion function in a small ball}

In this section, we show Theorem~\ref{torsion:ball}. We begin with some existence and qualitative properties for the torsion function.

\begin{proposition}\label{prop:tor}
Let $N\geq 1$, $r>0$ be such that $|B_r|<2^N e^{\frac{N}{2}(\Psi(\frac{N}{2})-\gamma)}|B_1|$, and consider the following problem
\begin{align}\label{eq:tau_lem}
    L_\Delta \tau = 1\quad \text{ in }B_r,\qquad \tau=0\quad \text{ in }\R^N \backslash B_r.
\end{align}
There is a unique classical positive solution $\tau$ to~\eqref{eq:tau_lem}. In particular, $\tau$ is radially symmetric and continuous in $\R^N$.
\end{proposition}

\begin{proof}
Under the assumption on $r$, by Theorem~\ref{mp:thm}, we have that $L_{\Delta}$ satisfies the maximum principle in $B_r$.  By \cite[Theorem 1.1]{CS22}, there is a classical solution of~\eqref{eq:tau_lem}, namely,~\eqref{eq:tau_lem} holds pointwisely.  Since $L_{\Delta}$ satisfies the maximum principle in $B_r$, we have that $\tau \geq 0$ in $B_r$. Moreover, $\tau>0$ in $B_r$, because, if $\tau(y_0)=0$ for some $y_0\in \Omega$, then we would have that 
\begin{align*}
1=L_\Delta \tau(y_0)=-c_N\int_{B}\frac{\tau(y)}{|y_0-y|^N}\, dy<0,
\end{align*}
which is absurd. Finally, $\tau$ is radially symmetric by the uniqueness of the solution. 
\end{proof}

\subsubsection{1D case}
We show the one-dimensional case first.

\begin{theorem} \label{torsion:1D}
Let $r>0$ be such that $r<2e^{\frac{1}{2}(\Psi(\frac{1}{2})-\gamma)}\approx 0.561459$ and let $\tau$ be the torsion function of the interval $(0,r)$, namely, the unique weak solution of 
\begin{align*}
    L_\Delta \tau = 1\quad \text{ in }(0,r),\qquad \tau=0\quad \text{ in }\R\setminus (0,r).
\end{align*}
    Then there is $c\in(0,1)$ such that
    \begin{align*}
\frac{\ell^{\frac{1}{2}}(\delta(x))}{c}\geq \tau(x)\geq c\,\ell^{\frac{1}{2}}(\delta(x))\quad \text{ for }x\in (0,r),
\end{align*}
where $\delta(x)=\dist(x,\partial (0,r))=\dist(x,\{0,r\})=\min\{x,r-x\}.$
\end{theorem}
\begin{proof}
The upper bound follows from Theorem~\ref{convex:thm}.  For the lower bound, let $u$ be given by~\eqref{u}.  Then $L_\Delta u\in L^\infty(\Omega)$ with $\Omega=(0,2)$, by Theorem~\ref{thm:L}. Let $r>0$ be as in the statement, $\eps:=\frac{1}{4}r$, $\widetilde \Omega:=(0,2\eps)\subset (0,r)$, and let $\widetilde u:\widetilde \Omega\to \R$ be given by $\widetilde u(x):=u(\frac{x}{\eps})$.  By Lemma~\ref{scaling:lem}, $L_\Delta \widetilde u\in L^\infty(\widetilde \Omega)$.

By Proposition~\ref{prop:tor}, $\tau>0$ in $[0,r)$ and $\tau\in C(\R)$.  Let $U:=\tau-\frac{1}{\|L_\Delta \widetilde u\|_\infty}\widetilde u$, then $L_\Delta U \geq 0$ in $\widetilde \Omega$ and $U\geq 0$ on $\R\backslash \widetilde \Omega$. By Theorem~\ref{mp:thm}, we obtain that $U\geq 0$ in $\R$, namely
\begin{align*}
    \tau(x)\geq C \widetilde u(x)=\frac{C}{\sqrt{-\ln \frac{x}{2\epsilon}}}=C\ell^\frac{1}{2}\left(\frac{x}{2\epsilon}\right)\geq C_1\ell^{\frac{1}{2}}(x)\quad \text{ for }x\in (0,\tfrac{\epsilon}{10})
\end{align*}
for some constants $C,C_1>0$, by Lemma~\ref{prop1}. One can argue similarly in a neighborhood of $r$ to obtain that $\tau(x)\geq C_2\ell^{\frac{1}{2}}(r-x)$  for $x\in (r-\tfrac{\epsilon}{10},r)$ and for some constant $C_2>0$. Let $I(\eps):=(\tfrac{\epsilon}{10},r-\tfrac{\epsilon}{10})$. Since $\tau>0$ in $(0,r)$, then
\begin{align*}
\tau\geq C_3\ell^\frac{1}{2}(\delta(x))\quad \text{ for }x\in I(\eps),\qquad C_3:=\frac{\inf_{I(\eps)} \tau}{\sup_{I(\eps)}\ell^\frac{1}{2}(\delta(x))}.
\end{align*}
The lower bound now follows with $c:=\min\{C_1,C_2,C_3\}$. This ends the proof.
\end{proof}

\subsubsection{The general case}

We are ready to show Theorem~\ref{torsion:ball}. 
\begin{proof}[Proof of Theorem~\ref{torsion:ball}]
The case $N=1$ is shown in Theorem~\ref{torsion:1D}. Assume $N\geq 2$, let $r$ be as in the statement, and let $v$ be as in Proposition~\ref{v:barrier:prop}. Let $D:=\{x\in\R^N\::\: |x-\frac{1}{2}|<\frac{1}{2},\, x_1>\frac{1}{2}\}\subset B_1$, $\widetilde v(x):=v(\frac{x}{r})$, and $\widetilde D:=rD\subset B_r$.  Then, by Lemma~\ref{scaling:lem} and Proposition~\ref{v:barrier:prop}, we have that $\widetilde v\in \cL^{\frac{1}{2}}(\R^N)\cap C^\infty(\widetilde D)\cap {\mathbb H}(\widetilde D)$ is a weak solution of $L_\Delta \widetilde v= f$ in $\widetilde D$ for some $f\in L^\infty(\widetilde D)\backslash\{0\}.$ Moreover, there is $c>0$ and $\delta>0$ such that
\begin{align}\label{v:est:2}
\widetilde v(x)>c\,\ell^{\frac{1}{2}}(\operatorname{dist}(x,\partial \widetilde D))\qquad \text{ for all }x\in \widetilde D\cap r(B_\delta(e_1)).
\end{align}

By Proposition~\ref{prop:tor}, $\tau\in C(\R^N)$ and $\tau>0$ in $B_r$.  Let $U:=\tau-k\widetilde v$ with $k:=1/\|f\|_{L^\infty(\widetilde D)}$. Then $L_\Delta U \geq 0$ (weakly) in $\widetilde D$ and $U=\tau\geq 0$ on $\R^N\backslash \widetilde D$. By Theorem~\ref{mp:thm}, we obtain that $U\geq 0$ in $\R^N$, namely, $\tau(x)\geq k\widetilde v(x)$ for $x\in \widetilde D.$ Then, by~\eqref{v:est:2},
\begin{align}
\liminf_{t\to 0^+} \frac{\tau(re_1-t e_1)}{\ell^\frac{1}{2}(t)}
&\geq k\liminf_{t\to 0^+} \frac{\widetilde v(re_1-t e_1)}{\ell^\frac{1}{2}(t)}
\geq k\liminf_{t\to 0^+} \frac{c\,\ell^{\frac{1}{2}}(\operatorname{dist}(re_1-t e_1,\partial \widetilde D))}{\ell^\frac{1}{2}(t)}\notag\\
&=k\liminf_{t\to 0^+} \frac{c\,\ell^{\frac{1}{2}}(\operatorname{dist}(re_1-t e_1,re_1))}{\ell^\frac{1}{2}(t)}
=k\liminf_{t\to 0^+} \frac{c\,\ell^{\frac{1}{2}}(t)}{\ell^\frac{1}{2}(t)}
=kc>0.\label{le:tau}
\end{align}
Since $\tau$ is radially symmetric, then the lower bound in~\eqref{liminf} follows. The lower estimate in~\eqref{tau:bds} now follows from~\eqref{le:tau} (with a proof by contradiction, for instance). On the other hand, the upper bound in~\eqref{liminf} and in~\eqref{tau:bds} holds by Theorem~\ref{convex:thm}.
\end{proof}

\subsection{A Hopf-type lemma for the logarithmic Laplacian}

For a bounded open set $\Omega$, the space ${\mathcal V}(\Omega)$ is defined in~\eqref{Vspace:def}.  We say that a function $v\in {\mathcal V}(\Omega)$ solves weakly that $L_\Delta v\geq 0$ in $\Omega$ if
\begin{align*}
\cE_L(v, \varphi) \geq 0\qquad \text{ for all $\varphi \in C^\infty_c(\Omega)$ with $\varphi \geq 0$ in $\Omega$}.
\end{align*}
We remark that $\cE_L(v, \varphi)<\infty$ for $v\in {\mathcal V}(\Omega)$ and $\varphi \in C^\infty_c(\Omega)$  by \cite[Lemma 4.4]{CW19}.  Recall that $\eta$ denotes the outer unit normal vector field along $\partial \Omega$. We are ready to show Theorem~\ref{strong:mp}.

\begin{proof}[Proof of Theorem~\ref{strong:mp}]
Let $x_0\in\partial \Omega$ and $v$ be as in the statement. By continuity and because $v\neq 0$, there are $\delta>0$, an open set $V\subset\{x\in\Omega~:~v(x)>\delta\}$, and $r>0$ such that $v$ solves weakly that $L_\Delta v\geq 0$ in $\mathcal B:=B_{r}(x_{0}-r\eta(x_0))\subset \Omega$ and $\operatorname{dist}(\mathcal B,V)>0.$  Note that $x_0\in\partial \mathcal B$.

By Theorem~\ref{mp:thm}, we can consider, if necessary, $r$ smaller so that $L_{\Delta}$ satisfies the weak maximum principle in $\mathcal B$. By Proposition~\ref{prop:tor}, there is a classical solution $\tau$ of
\begin{align*}
L_{\Delta}\tau=1\quad\mbox{in}~\mathcal B,\qquad 
\tau=0\quad\mbox{in}~\mathbb{R}^{N}\setminus \mathcal B.
\end{align*}
Moreover, $\tau>0$ in $\mathcal B$. Now we argue as in \cite{JF}. Let $\chi_{V}$ denote the characteristic function of $V$ and note that, for $x\in \mathcal B$, $\chi_{V}(x)=0$ and, therefore, 
\begin{align*}
L_{\Delta}\chi_{V}(x)=-c_{N}\int_{\mathbb{R}^{N}}\frac{\chi_{V}(y)}{|x-y|^{N}}\, dy=-c_{N}\int_{V}\frac{1}{|x-y|^{N}}\, dy\leq -c_{N}|V|\operatorname{diam}(\Omega)^{-N}.
\end{align*}
Let $K:=c_{N}|V|\operatorname{diam}(\Omega)^{-N}$ and 
\begin{align*}
\varphi:=k\left(\frac{K}{2}\tau+\chi_{V}\right)\qquad \text{ with }k:=\frac{1}{\frac{K}{2}\|\tau\|_{L^\infty(\mathcal B)}+1}.    
\end{align*}
Then, $L_{\Delta}\varphi\leq k(K/2-K)\leq 0$ in $\mathcal B$. Moreover, since $v>\delta$ in $V$, we have that $v-\delta\varphi$ solves weakly that
\begin{align*}
L_{\Delta}(v-\delta\varphi)\geq 0~\mbox{in}~\mathcal B\quad \text{ and }\quad 
v-\delta\varphi\geq 0~\mbox{in}~\mathbb{R}^{N}\setminus \mathcal B.
\end{align*}
Then, by Theorem~\ref{mp:thm}, $v\geq\delta\varphi\geq k\frac{K}{2}\delta\tau>0$ in $\mathcal B$. But then, for $t\in(0,r)$,
\begin{align*}
\frac{v(x_0-t\eta(x_0))}{\ell^\frac{1}{2}(t)}
\geq k\frac{K}{2} \delta\frac{\tau(x_0-t\eta(x_0))}{\ell^\frac{1}{2}(t)}
\geq c>0
\end{align*}
for some $c>0$, by Theorem~\ref{torsion:ball}, as claimed. 
\end{proof}

\begin{corollary}\label{cor:hopf}
Let $\Omega\subset \R^N$ be an open bounded set of class $C^2$ and let $v\in C(\mathbb{R}^{N})$ be such that $v=0$ on $\R^N\backslash \Omega$, $v>0$ in $\Omega$ and there is $c\in(0,1)$ with
\begin{align}\label{hopf:hyp}
c^{-1}&>
\lim\sup_{t\to 0^+} \frac{v(x_0-t\eta(x_0))}{\ell^\frac{1}{2}(t)}
\geq 
\lim\inf_{t\to 0^+} \frac{v(x_0-t\eta(x_0))}{\ell^\frac{1}{2}(t)}>c,
\end{align}
where $\eta(x_0)$ is a unitary exterior normal vector at $x_0$. Then there is $C>1$ such that 
\begin{align*}
C^{-1}\ell^{\frac{1}{2}}(\delta(x))<
v(x)
<C\ell^{\frac{1}{2}}(\delta(x))
\qquad \text{ for all }x\in \Omega,
\end{align*}
where $\delta(x):=\dist(x,\partial \Omega)$.
\end{corollary}
\begin{proof} This is a standard consequence of Hopf-type lemmas. We include a proof for completeness. We show first that
\begin{align}\label{aff1}
v(x)>c\, \ell^{\frac{1}{2}}(\delta(x))\qquad \text{ for all }x\in \Omega
\end{align}
for some $c>0.$  Indeed, assume by contradiction that there are $(x_n)_{n\in\N}\subset\Omega$ with
\begin{align}\label{chyp}
v(x_n)<\frac{1}{n}\, \ell^{\frac{1}{2}}(\delta(x_n))\qquad \text{ for all }x\in \Omega.
\end{align}
Since $v>0$ in $\Omega$, then (up to a subsequence) $x_n\to x_*$ with $x^*\in\partial \Omega$ as $n\to \infty.$  Let $y_n\in \partial \Omega$ be such that $x_n=y_n-\delta(x_n)\eta(y_n)$. By~\eqref{hopf:hyp},
\begin{align*}
\liminf_{n\to\infty}\frac{v(x_n)}{\ell^{\frac{1}{2}}(\delta(x_n))}\geq c>0;
\end{align*}
but this contradicts~\eqref{chyp} and thus~\eqref{aff1} follows.  The fact that $v(x)<C\, \ell^{\frac{1}{2}}(\delta(x))$ for all $x\in \Omega$ follows from Theorem~\ref{convex:thm}.
\end{proof}

\subsection{Uniqueness of positive solutions of logarithmic sublinear problems}\label{sublin:sec}
 
 Let $\Omega$ be a bounded set of class $C^2$; in particular $\Omega$ satisfies uniform exterior and interior sphere conditions.  For $\mu>0$, consider the problem 
\begin{align}\label{s}
L_{\Delta}u=-\mu\ln(|u|)u\quad\mbox{in}~\Omega,\qquad v=0\quad \text{ in }\R^N\backslash \Omega,
\end{align}

We say that $u\in\mathbb{H}(\Omega)$ is a weak solution of~\eqref{s} if 
\begin{align*}
\mathcal{E}_{L}(u,v)=-\mu\int_{\Omega}uv\ln(|u|)\, dx\quad\mbox{for all}~v\in\mathbb{H}(\Omega).
\end{align*}

In \cite[Theorem 1.1]{AS22}, the following result is shown. Recall that $\ell(r)=|\ln(\min\{r,\tfrac{1}{10}\})|^{-1}$.

\begin{theorem}\label{exisleast:intro}
For every $\mu>0$ there is a unique (up to a sign) least-energy weak solution $u\in\mathbb H(\Omega)$ of~\eqref{s} which is a global minimizer of the energy functional
\begin{align}\label{J0def}
J_{0}:\mathbb{H}(\Omega)\rightarrow\mathbb{R},\quad J_{0}(v):=\frac{1}{2}\mathcal{E}_{L}(v,v)+I(v),\quad I(v):=\frac{\mu}{4}\int_{\Omega}v^{2}\left(\ln(v^{2})-1\right)\, dx.
\end{align}
Moreover, $0<|u(x)|\leq (R^{2}e^{\frac{1}{2}-\rho_{N}})^{\frac{1}{\mu}}$ for $x\in \Omega$, where $R:=2\operatorname{diam}(\Omega)$ and $\rho_N$ is an explicit constant given in~\eqref{rhon}. Furthermore, $u\in C(\R^N)$, and there are $\alpha\in (0,1)$ and $C>0$ such that
\begin{align}\label{lhr}
 \sup_{\substack{x,y\in\R^N \\ x\neq y}}\frac{|u(x)-u(y)|}{\ell^\alpha(|x-y|)}<C.
\end{align}
\end{theorem}

In this theorem, the uniqueness (up to a sign) of the least energy solution is shown using a convexity-by-paths argument as in \cite[Section 6]{BFMST18}.  In particular, the following is shown: Given $u$ and $v$ in $\mathbb H(\Omega)$ such that $u^2\neq v^2$, let 
\begin{align}\label{gamma}
\theta(t,u,v):=((1-t)u^2+ t v^2)^\frac{1}{2},    
\end{align}
then
\begin{align}\label{claim}
\text{the function $t\mapsto J_0(\theta(t,u,v))$ is strictly convex in $[0,1]$,}    
\end{align}
Since a strictly convex function cannot have two global minimizers,~\eqref{claim} immediately yields the uniqueness of least-energy solutions.

Now we use~\eqref{claim} and Theorem~\ref{strong:mp} to yield the uniqueness of \emph{positive} solutions (which has to be the positive least-energy solution obtained in Theorem~\ref{exisleast:intro}).
\begin{theorem}\label{u:thm}
For every $\mu>0$ there is only one positive and bounded solution of~\eqref{s}.
\end{theorem}
\begin{proof}
Let $A$ be the set of positive and bounded critical points of $J_0$. By Theorem~\ref{exisleast:intro} we know that $A$ is nonempty. Let $u,v\in A$. Since $u,v\in L^\infty(\Omega)$, it follows that $\ln|u|u,\ln|v|v\in L^\infty(\Omega)$, and, by \cite[Theorem 1.11]{CW19}, we have that $u,v\in C(\overline{\Omega})$.  Then there is $\eps>0$ such that $-\mu\ln|u|u>0$ and $-\mu\ln|v|v>0$ in $\Omega_\eps:=\{x\in\Omega\::\: \dist(x,\partial \Omega)<\eps\}$.  Since $u$ and $v$ are continuous weak solutions of~\eqref{s}, we have in particular that $u$ and $v$ solve weakly 
\begin{align*}
    L_\Delta u\geq 0,\ L_\Delta v\geq 0\quad \text{ in }\Omega_\eps,\qquad u\geq 0,\ v\geq 0\quad \text{ in } \R^N\backslash \Omega_\eps.
\end{align*}
By Theorem~\ref{strong:mp}, there is $c\in(0,1)$ such that 
\begin{align*}
c^{-1}&>
\limsup_{t\to 0^+} \frac{u(x_0-t\eta(x_0))}{\ell^\frac{1}{2}(t)}\geq 
\liminf_{t\to 0^+} \frac{u(x_0-t\eta(x_0))}{\ell^\frac{1}{2}(t)}
>c,\\
c^{-1}&>
\limsup_{t\to 0^+} \frac{v(x_0-t\eta(x_0))}{\ell^\frac{1}{2}(t)}
\geq 
\liminf_{t\to 0^+} \frac{v(x_0-t\eta(x_0))}{\ell^\frac{1}{2}(t)}>c,
\end{align*}
 for all $x_0\in\partial \Omega.$ Then, by Corollary~\ref{cor:hopf}, we obtain that $u$ and $v$ are comparable, namely, that there is $M>1$ such that
 \begin{align}\label{comp}
     M>\frac{v(x)}{u(x)}>M^{-1}\qquad \text{ for }x\in \Omega.
 \end{align}

Let $\theta$ be as in~\eqref{gamma}. We now argue as in \cite[Theorem 6.1]{BFMST18}. Let $w:=\frac{v}{u} \chi_{\Omega},$ $z(\xi, t):=\frac{\xi^2-1}{\left(1-t+t \xi^2\right)^{\frac{1}{2}}+1}$, and let $\chi_{\Omega}$ denote the characteristic function of $\Omega$, then
\begin{align*}
\frac{\theta(t)-\theta(0)}{t}
=\frac{\theta(t)^2-u^2}{t(\theta(t)+u)}
=\frac{(1-t)u^2+ t v^2-u^2}{t(\theta(t)+u)}
=\frac{v^2-u^2}{\theta(t)+u}=u \frac{w^2-1}{\left(1-t+t w^2\right)^{\frac{1}{2}}+1}=u z(w, t).
\end{align*}
By~\eqref{comp}, for $x, y \in \Omega$,
\begin{align*}
u(x) z(w(x), t)-u(y) z(w(y), t)= & (u(x)-u(y)) z(w(x), t)-u(y)(z(w(y), t)-z(w(x), t))
\end{align*}
and, by the Mean-Value Theorem,
\begin{align*}
u(y)|z(w(x), t)-z(w(y), t)| 
&\leq C_1 u(y)|w(x)-w(y)|
= C_1 \left|u(y)\frac{v(x)}{u(x)}-v(y)\right|
\\
&=C_1 \left|v(x)-v(y)+\frac{v(x)}{u(x)}(u(y)-u(x))\right|\\
&\leq C_2(|v(x)-v(y)|+|u(y)-u(x)|),
\end{align*}
where $C_1:=\sup _{(k,t) \in\left[0,M\right]\times[0,1]}\left|\partial_\xi z(k, t)\right|<\infty$ and $C_2:=C_1+M$.

On the other hand, if $x \in \Omega$ and $y \in \mathbb{R}^N \backslash \Omega$, then $u(y)=0$ and
\begin{align*}
|u(x) z(w(x), t)-u(y) z(w(y), t)| \leq C_1 M|u(x)| = C_1 M|u(x)-u(y)| .
\end{align*}

Then there is $C_3>0$ such that
\begin{align*}
\left\|\frac{\theta(t)-\theta(0)}{t}\right\|^2
&=\int_{\R^N}\int_{B_1(x)}\frac{|\left(\frac{\theta(t)-\theta(0)}{t}\right)(x)-\left(\frac{\theta(t)-\theta(0)}{t}\right)(y)|^2}{|x-y|^N}\,dy\,dx\\
&\leq C_3\left( \int_{\R^N}\int_{B_1(x)}\frac{|u(x)-u(y)|^2}{|x-y|^N}\,dy\,dx
+\int_{\R^N}\int_{B_1(x)}\frac{|v(x)-v(y)|^2}{|x-y|^N}\,dy\,dx\right)\\
&=C_3(\|u\|^2+\|v\|^2).
\end{align*}

This, together with~\eqref{claim}, guarantees that all the assumptions of \cite[Theorem 1.1]{BFMST18} are satisfied. Then, this result implies that $A$ has at most one element, and this ends the proof. 
\end{proof}

\appendix

\section{Some auxiliary results}

\subsection{Scaling properties of the logarithmic Laplacian}\label{scal:sec}

In this section, for completeness, we show some easy scaling properties for the logarithmic Laplacian (see also \cite[Lemma 2.5]{LW21}).  For $\lambda>0$ and a function $\varphi:\R^N\to\R$, let
\begin{align*}
\varphi_\lambda(x):=\varphi(\lambda^{-1} x).
\end{align*}

Let $\Omega$ be an open bounded Lipschitz subset of $\R^N$. We set $\Omega_\lambda:=\lambda \Omega=\{\lambda x\::\: x\in \Omega\}$.

\begin{lemma}[On smooth functions]\label{0:sc:lem}
Let $u\in C_c^\infty(\Omega)$ be a solution of $L_\Delta u = f$ in $\Omega,$ $u=0$ on $\R^N\backslash \Omega$ for some $f\in L^\infty(\Omega)$. Then, for any $\lambda>0$, $u_\lambda$ is a solution of $L_\Delta u_\lambda =\widetilde f_\lambda$ in $\Omega_\lambda$, $u_\lambda=0$ on $\R^N\backslash \Omega_\lambda,$ with $\widetilde f_\lambda:= f_\lambda+\ln(\lambda^{-2}) u_\lambda\in L^\infty(\Omega_\lambda)$.
\end{lemma}
\begin{proof}
Let $x\in  \Omega_\lambda=\lambda \Omega$, then $(-\Delta)^s u_\lambda(x) = \lambda^{-2s}(-\Delta)^s u(\lambda^{-1}x)$ and then
\begin{align*}
    L_\Delta u_\lambda 
    = \partial_s\mid_{s=0}(-\Delta)^s u_\lambda(x) 
    = \partial_s\mid_{s=0} \lambda^{-2s}(-\Delta)^s u(\lambda^{-1}x)
    = \ln(\lambda^{-2}) u(\lambda^{-1}x)+L_\Delta  u(\lambda^{-1}x),
\end{align*}
where we used \cite[Theorem 1.1]{CW19} to justify the first and last equalities.
\end{proof}

The next result is an integration by parts formula under slightly weaker assumptions than those in \cite[equation (3.11)]{CW19}.
\begin{lemma}[Integration by parts] \label{ibyp:lem} Let $\Omega\subset \R^N$ be a bounded Lipschitz set and let $w\in \mathbb{H}(\Omega)$ be such that $L_{\Delta} w\in L^\infty(\Omega)$. Then,
\begin{align*}
\int_{\mathbb{R}^N}\left[L_{\Delta} w\right] v\, dx=\mathcal{E}_L(w, v)\qquad \text{for all $v \in \mathbb{H}(\Omega)$.}
\end{align*}
\end{lemma}
\begin{proof}
A similar result for uniformly Dini continuous functions can be found in \cite[equation (3.11)]{CW19}, and the same arguments can be extended for $w$ as in the statement. We give a proof for completeness. Let $\mathbf{k}: \mathbb{R}^N \backslash\{0\} \rightarrow \mathbb{R}$ and $\mathbf{j}: \mathbb{R}^N \rightarrow \mathbb{R}$ be given by
\begin{align*}
\mathbf{k}(z):=c_N 1_{B_1}(z)|z|^{-N}\qquad \text{ and }\qquad    \mathbf{j}(z):=c_N 1_{\mathbb{R}^N \backslash B_1}(z)|z|^{-N}.
\end{align*}
Then, for $x\in \Omega$,
$$
L_{\Delta} w(x)=\int_{\mathbb{R}^N}(w(x)-w(y)) \mathbf{k}(x-y) d y-[\mathbf{j} * w](x)+\rho_N w(x)
$$
and, by a standard argument using a change of variables, 
\begin{align*}
\int_{\mathbb{R}^N}\left[L_{\Delta} w\right] v\, dx
&=\int_{\mathbb{R}^N}\left(\int_{\mathbb{R}^N}(w(x)-w(y)) \mathbf{k}(x-y) d y-[\mathbf{j} * w](x)+\rho_N w(x)\right) v(x)\, dx\\
&=\frac{1}{2} \int_{\mathbb{R}^N} \int_{\mathbb{R}^N}(w(x)-w(y))(v(x)-v(y)) \mathbf{k}(x-y) \,dx\, dy-\int_{\mathbb{R}^N}\left(\mathbf{j} * w-\rho_N w\right) v \, dx\\
&=\mathcal{E}_L(w, v)
\end{align*}
for all $v \in \mathbb{H}(\Omega)$, as claimed.
\end{proof}

\begin{lemma}[On weak solutions]\label{scaling:lem}
Let $u\in \cH(\Omega)\cap L^\infty(\Omega)$ be a weak solution of $L_\Delta u = f$ in $\Omega,$ $u=0$ on $\R^N\backslash \Omega$ for some $f\in L^\infty(\Omega)$. Then, for any $\lambda>0$, $u_\lambda$ is a weak solution of $L_\Delta u_\lambda =\widetilde f_\lambda$ in $\lambda\Omega$, $u_\lambda=0$ on $\R^N\backslash \lambda\Omega,$ with $\widetilde f_\lambda:= f_\lambda+\ln(\lambda^{-2}) u_\lambda\in L^\infty(\Omega)$.
\end{lemma}
\begin{proof}
Let $\varphi\in C_c^\infty(\Omega)$. Using Lemmas~\ref{0:sc:lem} and~\ref{ibyp:lem},
\begin{align*}
\cE_L(u_\lambda,\varphi_\lambda)
&=\int_{\R^N} u_\lambda L_\Delta\varphi_\lambda\, dx
=\int_{\R^N} u(\lambda^{-1}x) L_\Delta\varphi(\lambda^{-1}x)+\ln(\lambda^{-2})  u(\lambda^{-1}x)\varphi(\lambda^{-1}x)\, dx\\
&=\lambda^{N}\int_{\R^N} u L_\Delta\varphi+\ln(\lambda^{-2})  u\varphi\, dx
=\lambda^{N}\left(\cE_L(u,\varphi)+\ln(\lambda^{-2})\int_{\R^N} u\varphi\, dx\right)\\
&=\int_{\R^N}(f_\lambda +\ln(\lambda^{-2}) u_\lambda)\varphi_\lambda\, dx.
\end{align*}
\end{proof}

Now, a scaling property for pointwise solutions easily follows from the previous result.

\begin{lemma}[On pointwise solutions]\label{scaling:lem:class}
Let $w$ and $\Omega$ be as in Lemma~\ref{ibyp:lem}.  Moreover, assume that $w\in L^\infty(\Omega)$ and let $f(x):=L_\Delta w(x)$ for $x\in \Omega$. Then, for any $\lambda>0$, $w_\lambda$ is a solution of $L_\Delta w_\lambda =\widetilde f_\lambda$ in $\Omega_\lambda$, $w_\lambda=0$ on $\R^N\backslash \Omega_\lambda,$ with $\widetilde f_\lambda:= f_\lambda+\ln(\lambda^{-2}) w_\lambda\in L^\infty(\Omega_\lambda)$.
\end{lemma}
\begin{proof}
Let $\lambda>0$, $w$, and $\Omega$ as in the statement, $\varphi\in C^\infty_c(\Omega_\lambda)$, and let $\psi(x):=\varphi(\lambda x)$.  Then, by Lemmas~\ref{0:sc:lem},~\ref{ibyp:lem}, and a change of variables,
\begin{align*}
\lambda^{-N}\int_{\Omega_\lambda} f_\lambda \varphi\, dx
&=\int_\Omega f \psi\, dx=\int_\Omega L_\Delta w \psi\, dx=\mathcal{E}_L(w,\psi)=\int_\Omega w L_\Delta\psi\, dx\\
&=\int_\Omega w(x) (L_\Delta\varphi(\lambda x)+\ln(\lambda^2)\varphi(\lambda x))\, dx=\lambda^{-N}\int_{\Omega_\lambda} w_\lambda(x) (L_\Delta\varphi(x)+\ln(\lambda^2)\varphi(x))\, dx\\
&=\lambda^{-N}\int_{\Omega_\lambda} (L_\Delta w_\lambda(x)+\ln(\lambda^2)w_\lambda(x))\varphi(x)\, dx.
\end{align*}
As a consequence, $\int_{\Omega_\lambda} (L_\Delta w_\lambda-(f_\lambda+\ln(\lambda^{-2})w_\lambda))\varphi\, dx=0.$ Since $\varphi$ is arbitrary, the claim follows.
\end{proof}

\subsection{A Leibniz-type formula for the logarithmic Laplacian}\label{Leib:sec}

Let $E$ be a bounded measurable set of $\mathbb R^N$ and $u:E\to \mathbb{R}$ be a measurable function. Recall that the modulus of continuity of $u$ at a point $x\in E$ is defined as
\begin{align*}
    \omega_{u,x,E}:(0,\infty)\to[0,\infty), \qquad \omega_{u,x,E}(r)=\sup_{
    \genfrac{}{}{0pt}{2}{y\in E,}{|y-x|\leq r}
    }|u(y)-u(x)|,
\end{align*}
and $u$ is said to be Dini continuous at $x$ if $\int_{0}^{1}\frac{\omega_{u,x,E}(r)}{r}\,dr<\infty$. Let 
\begin{align*}
 L_0^1(\mathbb R^N):=\left\{
 u\in L^1_{loc}(\R^N)\::\:\int_{\mathbb R^N}\frac{|u(x)|}{(1+|x|)^{N}}dx < \infty
 \right\}.
\end{align*}

\begin{lemma}
    Let $u,v\in L_0^1(\mathbb R^N)$ be such that $u$ and $v$ are Dini continuous functions at some $x\in \mathbb R^N$ and $v\in L^\infty(\mathbb R^N)$. Then $L_{\Delta}[uv](x)$ is well defined (in the sense of formula~\eqref{loglap:def}) and 
    \begin{align}\label{eq:leib_log}
        L_{\Delta}[uv](x)= u(x)L_{\Delta}v(x)+v(x) L_{\Delta}u(x)-I(u,v)(x),
    \end{align}
    where
    \begin{align}\notag
        I(u,v)(x)&=c_{N}\int_{B_{1}(x)}\frac{(u(x)-u(y))(v(x)-v(y))}{|x-y|^N}dy\\ 
        &\quad +c_{N} \int_{\R^N\setminus B_1(x)}\frac{u(y)v(y)-u(x)v(y)-u(y)v(x)}{|x-y|^N}dy
        +\rho_{N}u(x)v(x). \label{eq:def_I}
    \end{align}
\end{lemma}

\begin{proof}
Let $x\in \Omega$ be as in the statement.  First we show that the right hand side of~\eqref{eq:leib_log} is well defined. The first two terms, $u(x)L_{\Delta}v(x)$ and $v(x) L_{\Delta}u(x)$ are well defined by our assumptions on $u$ and $v$, see \cite[Proposition 2.2]{CW19}. On the other hand, since $v\in L^\infty(\R^N)$ and $u$ is Dini continuous at $x$,
\begin{align*}
    \int_{B_1(x)}&\frac{(u(x)-u(y))(v(x)-v(y))}{|x-y|^N}dy = \int_{B_1}\frac{(u(x)-u(z+x))(v(x)-v(z+x))}{|z|^N}\,dz \\
    &\leq 2\|v\|_{L^\infty}\int_{B_1}\frac{|u(x)-u(x+z)|}{|z|^N}dz
    \leq 2 \|v\|_\infty|\mathbb S^{N-1}|\int_0^1\frac{\omega_{u,x,B_1}(r)}{r}\,dr<\infty.
\end{align*}
Furthermore, since $v\in L^\infty(\mathbb R^N)$ and $u\in L_0^1(\mathbb R)$, there is a constant $C_x>0$ only depending on $x$ such that
\begin{align*}
    \int_{R^{N}\setminus B_1(x)}\frac{u(y)v(y)}{|x-y|^N}dy \leq \|v\|_{L^\infty} \int_{\R^N\setminus B_1(x)}\frac{|u(y)|}{|x-y|^N}dy \leq C_{x}\|v\|_{L^\infty}\int_{\R^N}\frac{|u(y)|}{(1+|y|)^N}dy<\infty.
\end{align*}
Similarly, $\int_{R^{N}\setminus B_1(x)}\frac{u(x)v(y)}{|x-y|^N}dy$ and $\int_{R^{N}\setminus B_1(x)}\frac{u(y)v(x)}{|x-y|^N}dy$ are also finite. Then, since
\begin{align*}
L_{\Delta}[uv](x)=c_{N}\int_{B_{1}(x)}\frac{u(x)v(x)-u(y)v(y)}{|x-y|^{N}}\, dy-c_{N}\int_{\mathbb{R}^{N}\setminus B_{1}(x)}\frac{u(y)v(y)}{|x-y|^{N}}\, dy+\rho_{N}u(x)v(x),
\end{align*}
identity~\eqref{eq:leib_log} follows by noting that
\begin{align*}
    c_{N}&\int_{B_1(x)}\frac{u(x)v(x)-u(y)v(y)}{|x-y|^N}dy=u(x)\, c_{N}\int_{B_1(x)}\frac{v(x)-v(y)}{|x-y|^N}dy+c_{N}\int_{B_1(x)}\frac{(u(x)-u(y))v(y)}{|x-y|^N}dy \\
    &=u(x)\, c_{N}\int_{B_1(x)}\frac{v(x)-v(y)}{|x-y|^N}dy+v(x)\, c_{N}\int_{B_1(x)}\frac{u(x)-u(y)}{|x-y|^N}dy- c_{N}\int_{B_1(x)}\frac{(u(x)-u(y))(v(x)-v(y))}{|x-y|^N}dy.
\end{align*}

\end{proof}

\subsection*{Acknowledgements} We thank Héctor Chang-Lara for helpful discussions on this topic and we thank the anonymous referees for their helpful comments and suggestions that substantially improved the quality of our paper.

\bigskip

\bigskip
\begin{flushleft}
\textbf{Víctor Hernández-Santamaría and Alberto Saldaña}\\
Instituto de Matemáticas\\
Universidad Nacional Autónoma de México\\
Circuito Exterior, Ciudad Universitaria\\
04510 Coyoacán, Ciudad de México, Mexico\\
E-mails: \texttt{victor.santamaria@im.unam.mx, alberto.saldana@im.unam.mx} 
\vspace{.3cm}
\end{flushleft}

\smallskip

\begin{flushleft}
\textbf{Luis Fernando López Ríos}\\
Instituto de Investigaciones en Matemáticas Aplicadas y en Sistemas\\
Universidad Nacional Autónoma de México\\
Circuito Escolar s/n, Ciudad Universitaria\\
C.P. 04510, Ciudad de México, Mexico\\
E-mail: \texttt{luis.lopez@mym.iimas.unam.mx} 
\vspace{.3cm}
\end{flushleft}

\end{document}